\documentclass[a4paper]{amsart}
\usepackage{amssymb, enumitem}
\usepackage[all]{xy}
\usepackage{hyperref, aliascnt}
\usepackage{mathtools}

\newtheorem{lma}{Lemma}[section]

\newaliascnt{thmCt}{lma}
\newtheorem{thm}[thmCt]{Theorem}
\aliascntresetthe{thmCt}

\newaliascnt{corCt}{lma}
\newtheorem{cor}[corCt]{Corollary}
\aliascntresetthe{corCt}

\newaliascnt{prpCt}{lma}
\newtheorem{prp}[prpCt]{Proposition}
\aliascntresetthe{prpCt}

\newtheorem*{thm*}{Theorem}
\newtheorem*{cor*}{Corollary}
\newtheorem*{prop*}{Proposition}

\theoremstyle{definition}

\newaliascnt{pgrCt}{lma}
\newtheorem{pgr}[pgrCt]{}
\aliascntresetthe{pgrCt}

\newaliascnt{dfnCt}{lma}
\newtheorem{dfn}[dfnCt]{Definition}
\aliascntresetthe{dfnCt}

\newaliascnt{rmkCt}{lma}
\newtheorem{rmk}[rmkCt]{Remark}
\aliascntresetthe{rmkCt}

\newaliascnt{rmksCt}{lma}

\aliascntresetthe{rmksCt}

\newaliascnt{exaCt}{lma}
\newtheorem{exa}[exaCt]{Example}
\aliascntresetthe{exaCt}

\newaliascnt{qstCt}{lma}
\newtheorem{qst}[qstCt]{Question}
\aliascntresetthe{qstCt}

\newaliascnt{ntnCt}{lma}
\newtheorem{ntn}[ntnCt]{Notation}
\aliascntresetthe{ntnCt}

\def\today{\number\day\space\ifcase\month\or   January\or February\or
   March\or April\or May\or June\or   July\or August\or September\or
   October\or November\or December\fi\   \number\year}

\newcommand{\andSep}{\,\,\,\text{ and }\,\,\,}
\newcommand{\axiomO}[1]{(O#1)}
\newcommand{\CatCu}{\ensuremath{\mathrm{Cu}}}
\newcommand{\CuSgp}{$\CatCu$-sem\-i\-group}
\newcommand{\CuMor}{$\CatCu$-mor\-phism}
\newcommand{\soft}{{\rm{soft}}}
\newcommand{\NN}{{\mathbb{N}}}
\newcommand{\KK}{{\mathbb{K}}}
\newcommand{\id}{{\operatorname{id}}}
\newcommand{\Aut}{{\mathrm{Aut}}}
\newcommand{\NNbar}{\overline{\NN}}
\newcommand{\ca}{$\mathrm{C}^*$-algebra}

\DeclareMathOperator{\QT}{QT}
\DeclareMathOperator{\Cu}{Cu}
\DeclareMathOperator{\Lsc}{Lsc}
\DeclareMathOperator{\LAff}{LAff}
\DeclareMathOperator{\rc}{rc}

\newcounter{theoremintro}

\newaliascnt{thmIntroCt}{theoremintro}
\newtheorem{thmIntro}[thmIntroCt]{Theorem}
\aliascntresetthe{thmIntroCt}

\title{Ranks of soft operators in nowhere scattered $\mathrm{C}^*$-algebras}
\date{\today}

\author{M. Ali Asadi-Vasfi}
\address{M. Ali Asadi-Vasfi, 
Department of Mathematics, University of Toronto, Toronto, Ontario, Canada.
}
\urladdr{https://sites.google.com/view/maliasadivasfi/home?authuser=0}

\author{Hannes Thiel}
\address{Hannes~Thiel, 
Department of Mathematical Sciences, Chalmers University of Technology and University of
Gothenburg, Gothenburg SE-412 96, Sweden.}
\email{hannes.thiel@chalmers.se}
\urladdr{www.hannesthiel.org}

\author{Eduard Vilalta}
\address{Eduard Vilalta,
The Fields Institute for Research in Mathematical Sciences,
Toronto, Ontario M5T 3J1, Canada.}
\email{eduardvilaltavila@gmail.com}
\urladdr{www.eduardvilalta.com}

\thanks{
The second named author was partially supported by the Deutsche Forschungsgemeinschaft (DFG, German Research Foundation) under Germany's Excellence Strategy EXC 2044-390685587 (Mathematics M\"{u}nster: Dynamics-Geometry-Structure), by the ERC Consolidator Grant No.~681207, and by the Knut and Alice Wallenberg Foundation (KAW 2021.0140).
The third named author was partially supported by MINECO (grant No.\ PRE2018-083419 and No.\ PID2020-113047GB-I00), by the Comissionat per Universitats i Recerca de la Generalitat de Catalunya (grant No.\ 2021SGR01015), and by the Fields Institute for Research in Mathematical Sciences.
}


\subjclass[2010]%
{Primary
46L05; 
Secondary
06B35, 
06F05, 
19K14, 
46L35, 
46L80, 
46L85. 
}

\date{\today}

\begin{document}

\begin{abstract}
We show that for \ca{s} with the  Global Glimm Property, the rank of every operator can be realized as the rank of a soft operator, that is, an element whose hereditary sub-\ca{} has no nonzero, unital quotients.
This implies that the radius of comparison of such a \ca{} is determined by the soft part of its Cuntz semigroup.

Under a mild additional assumption, we show that every Cuntz class dominates a (unique) largest soft Cuntz class.
This defines a retract from the Cuntz semigroup onto its soft part, and it follows that the covering dimensions of these semigroups differ by at most $1$.
\end{abstract}

\maketitle

\section{Introduction}

Realizing every strictly positive, lower-semicontinuous, affine function on the tracial state space of a separable, simple, nuclear, non-elementary \ca{} as the rank of an operator in its stabilization is a deep and open problem, first studied in \cite{DadTom10Ranks}. 
A positive solution to this problem would imply that every separable, simple, non-elementary \ca{} of locally finite nuclear dimension and strict comparison of positive elements is $\mathcal{Z}$-stable, thus proving the remaining implication of the prominent Toms-Winter conjecture (\cite[Section~5]{Win18ICM}) in this case; 
see, for example, \cite[Section~9]{Thi20RksOps} and the discussion in \cite[Section~5]{CasEviTikWhi22UniPropGamma}.

When the \ca{} $A$ is not simple, the problem is still of much interest, but one needs to replace the tracial state space by the cone $\QT(A)$ of lower-semicontinuous, extended-valued $2$-quasitraces on~$A$.
Each such quasitrace extends canonically to the stabilization $A\otimes\KK$, and the \emph{rank} of an operator $a \in (A\otimes\KK)_+$ is defined as the map $\widehat{[a]}\colon\QT(A)\to[0,\infty]$ given by
\[
\widehat{[a]}(\tau) := d_\tau(a) := \lim_{n\to\infty}\tau(a^{1/n})
\]
for $\tau \in \QT(A)$.
The \emph{rank problem} is then to determine which functions on~$\QT(A)$ arise as the rank of a positive operator in $A$ or $A\otimes\KK$.

A natural obstruction arises if $A$ has a nonzero elementary ideal-quotient, that is, if there are closed ideals $I \subseteq J \subseteq A$ such that $J/I$ is $\ast$-isomorphic to $\KK(H)$ for some Hilbert space $H$.
In this case, the natural trace on $\KK(H)$ induces a quasitrace $\tau \in \QT(A)$ that is discrete in the sense that $d_\tau(a) \in \{0,1,2,\ldots,\infty\}$ for every $a \in (A\otimes\KK)_+$.
A similar obstruction arises in the representation of interpolation groups by continuous, affine functions on their state space;
see \cite[Chapter~8]{Goo86GpsInterpolation}.

To avoid this obstruction, it is therefore natural to assume that $A$ has no nonzero elementary ideal-quotients, a condition termed \emph{nowhere scatteredness} in \cite{ThiVil21arX:NowhereScattered}. 
Building on the results from \cite{Thi20RksOps}, the rank problem was solved in \cite{AntPerRobThi22CuntzSR1} for nowhere scattered \ca{s} that have stable rank one: 
Every function on $\QT(A)$ that satisfies the `obvious' conditions arises as the rank of an operator in $A\otimes\KK$;
see \cite[Theorem~7.13]{AntPerRobThi22CuntzSR1} for the precise statement.
Moreover, one can arrange for the operator to be \emph{soft}, which means that it generates a hereditary sub-\ca{} that has no nonzero unital quotients; 
see \cite[Definition~3.1]{ThiVil23arX:Soft}.

As a consequence, in a nowhere scattered, stable rank one \ca{}, the rank of every operator can be realized as the rank of a \emph{soft} operator. 

\medskip

The aim of this paper is to study this phenomenon in greater generality and, more concretely, to investigate when the rank of every operator in a \ca{} $A$ can be realized as the rank of a soft element. 
We show that this holds whenever~$A$ satisfies the \emph{Global Glimm Property} --- a notion conjectured to be equivalent to nowhere scatteredness; see \autoref{pgr:GGPNSCa}. Namely, we prove:

\begin{thmIntro}[\ref{prp:SoftRankCuGGP}]
\label{thmInt:SoftRanks}
Let $A$ be a stable \ca{} with the Global Glimm Property. 
Then, for any $a\in A_+$ there exists a soft element $b \in A_+$ with $b \precsim a$ and such that 
\[
d_\tau(a) = d_\tau(b)
\]
for every $\tau \in \QT(A)$.
\end{thmIntro}

In \autoref{thmInt:SoftRanks} above we use $\precsim$ to denote the \emph{Cuntz subequivalence}, a relation between positive elements introduced by Cuntz in \cite{Cun78DimFct}. 
This relation allows one to define the \emph{Cuntz semigroup}, an object that has played an important role in the structure and classification theory of \ca{s}; 
see \autoref{pgr:CuntzSgp} and \cite{CowEllIva08CuInv, Tom08ClassificationNuclear, Win12NuclDimZstable, Thi20RksOps, AntPerRobThi22CuntzSR1}. 
As explained in \autoref{pgr:CuSgps}, the study of the Cuntz semigroup has often come in hand with the development of abstract Cuntz semigroups, also known as \emph{\CuSgp{s}}; 
see \cite{Rob13Cone, AntPerThi20CuntzUltraproducts, AntPerThi20AbsBivariantCu, AntPerThi20AbsBivarII, Vil22LocCharAI, CanVil23arX:FraisseCu} among many others.

If an operator $a$ is soft, then its Cuntz class $[a]$ is strongly soft (we recall the definition at the beginning of \autoref{sec:Soft}).
If $A$ has the Global Glimm Property, then every strongly soft Cuntz class arises this way, and it follows that the submonoid $\Cu(A)_\soft$ of strongly soft Cuntz classes agrees with the subset of Cuntz classes with a soft representative;
see \autoref{prp:CharSoftCuClassGGP}.

The cone $\QT(A)$ is naturally isomorphic to the cone $F(\Cu(A))$ of functionals on the Cuntz semigroup $\Cu(A)$;
see \cite[Theorem~4.4]{EllRobSan11Cone}.
As an application of \autoref{thmInt:SoftRanks}, we show that the same is true for the cone of functionals on $\Cu(A)_\soft$.

\begin{thmIntro}[\ref{prp:QTAEqFCuASof}]
Let $A$ be a \ca{} with the Global Glimm Property. 
Then, $\QT(A)$ is naturally isomorphic to $F(\Cu (A)_\soft )$.
\end{thmIntro}

We introduce in \autoref{sec:SepCu} a weak notion of cancellation for Cuntz semigroups, which we term \emph{left-soft separativity}; see \autoref{dfn:LefSofSep}. Whenever a \ca{} with the Global Glimm Property has a left-soft separative Cuntz semigroup, the relation between arbitrary and soft elements from \autoref{thmInt:SoftRanks} can be made more precise:

\begin{thmIntro}[\ref{prp:sigmaSoft}, \ref{prp:SsoftGenCuMor}]\label{thmInt:RetRanks}
Let $A$ be a \ca{} with the Global Glimm Property. 
Assume that $\Cu (A)$ is left-soft separative.
Then;
\begin{itemize}
\item[(i)] 
For every element $x\in \Cu (A)$ there exists a greatest element in $\Cu (A)_\soft$ below $x$, denoted by $\sigma (x)$.
\item[(ii)] 
We have $\lambda (\sigma (x))=\lambda (x)$ for every $x \in \Cu(A)$ and $\lambda \in F( \Cu (A))$.
\item[(iii)]
The map $\sigma\colon \Cu (A)\to \Cu (A)_\soft$, defined by $x\mapsto \sigma (x)$, preserves order, suprema of increasing sequences, and is superadditive.
\end{itemize}
\end{thmIntro}
	
We show in \autoref{sec:SepCu} that the Cuntz semigroup is left-soft separative whenever the \ca{} has stable rank one or strict comparison of positive elements. 
Under these assumptions, we also show that $\sigma$ is subadditive and, consequently, a generalized Cu-morphism; 
see \autoref{prp:SigmaCuMorCStar}. 
Then $\Cu (A)_\soft$ is a retract of $\Cu (A)$, as defined in \cite{ThiVil22DimCu}.

Using structure results of retracts and soft elements, we study the covering dimension (\cite{ThiVil22DimCu}) and the radius of comparison (\cite{BlaRobTikTomWin12AlgRC}) of \ca{s} with the Global Glimm Property in terms of their soft elements.

\begin{thmIntro}[\ref{prp:SsoftDimGen}]\label{thmIntro:SsoftDim}
 Let $A$ be a \ca{} with the Global Glimm Property. Assume one of the following:
 \begin{itemize}
  \item[(i)] $A$ has strict comparison of positive elements;
  \item[(ii)] $A$ has stable rank one;
  \item[(iii)] $A$ has topological dimension zero, and $\Cu (A)$ is left-soft separative.
 \end{itemize}
 
 Then $\dim(\Cu (A)_\soft)\leq \dim(\Cu (A))\leq \dim(\Cu (A)_\soft)+1$.
\end{thmIntro}

\begin{thmIntro}[\ref{prp:SsoftRcCa}]\label{thmIntro:SsoftRC}
Let $A$ be a unital, separable \ca{} with the Global Glimm Property. 
Assume that $A$ has stable rank one. 
Then 
\[
\rc \big( \Cu(A),[1] \big) 
= \rc \big( \Cu(A)_\soft , \sigma ([1]) \big).
\]
\end{thmIntro}

We finish the paper with some applications of Theorems \ref{thmIntro:SsoftDim} and \ref{thmIntro:SsoftRC} to crossed products; see \autoref{prp:DimWTRP} and \autoref{exa:rcWTRP}.

\subsection*{Acknowledgements}

Part of this research was conducted during the Cuntz Semigroup Workshop 2021 at WWU Münster, and during the short research visit of the third author to the first author at the Institute of Mathematics of the Czech Academy of Sciences.
They thank both of the institutions for their hospitality. 

This work was completed while the authors were attending the Thematic Program on Operator Algebras and Applications at the Fields Institute for Research in Mathematical Sciences in August and September 2023, and they gratefully acknowledge the support and hospitality of the Fields Institute.

\section{Preliminaries}

In this section we recall definitions and results that will be used in the sections that follow.
The reader is referred to \cite{AraPerTom11Cu}, \cite{AntPerThi18TensorProdCu} and \cite{GarPer23arX:ModernCu} for an extensive introduction to the theory of \CuSgp{s} and their interplay with Cuntz semigroups.

Given a \ca{} $A$, we use $A_+$ to denote the set of its positive elements.

\begin{pgr}[The Cuntz semigroup]
\label{pgr:CuntzSgp}
Let $A$ be a \ca{}. 
Given $a, b \in A_+$, one says that $a$ is \emph{Cuntz subequivalent} to $b$, written $a\precsim b$, if there exists a sequence $(v_n)_n$ in $A$ such that $a=\lim_n v_n b v_n^*$. 
Further, one says that $a$ is \emph{Cuntz equivalent} to $b$, written $a\sim b$, if $a\precsim b$ and $b \precsim a$.

The \emph{Cuntz semigroup} of $A$, denoted by $\Cu (A)$, is the positively ordered monoid defined as the quotient $(A\otimes \mathcal{K})_+/{\sim }$ equipped with the order induced by $\precsim$ and the addition induced by addition of orthogonal elements.
For further details we refer to \cite{AraPerTom11Cu, AntPerThi18TensorProdCu, GarPer23arX:ModernCu}.
\end{pgr}

\begin{pgr}[\CuSgp{s}]
\label{pgr:CuSgps}
Let $(P, \leq)$ be a partially ordered set. 
Suppose that every increasing sequence in $P$ has a supremum. 
Given two elements $x,y$ in $P$, one says that $x$ is \emph{way-below} $y$, 
denoted $x\ll y$, if for every increasing sequence $(z_n)_n$ in $P$ satisfying $y\leq \sup_n z_n$, there exists some $m\in\NN$ such that $x\leq z_m$.

As defined in \cite{CowEllIva08CuInv}, a \emph{\CuSgp{}} is a positively ordered monoid~$S$ satisfying two domain-type conditions and two compatibility conditions:
\begin{enumerate}
\item[(O1)]
Every increasing sequence in $S$ has a supremum.
\item[(O2)]
For every element $x$ in $S$, there exists a sequence $(x_n)_{n}$ in $S$ such that $x_0 \ll x_1 \ll x_2 \ll \cdots$ and such that $x = \sup_n x_n$.  
\item[(O3)] 
The addition is compatible with the way-below relation, that is, for every $x', x, y', y  \in S$ satisfying $x' \ll x$ and $y' \ll y$, we have $x' + y' \ll x + y$.
\item[(O4)]
The addition is compatible with suprema of increasing sequences, that is, for every increasing sequences $(x_n)_n$ and $(y_n)_n$ in $S$, we have $\sup_n (x_n+y_n) = \sup_nx_n + \sup_ny_n$.
\end{enumerate} 

It follows from \cite{CowEllIva08CuInv} that the Cuntz semigroup of any \ca{} always satisfies \axiomO{1}-\axiomO{4}. 
Specifically, the Cuntz semigroup of any \ca{} is a \CuSgp{}.

Given a monoid morphism $\varphi$ between two \CuSgp{s}, we say that $\varphi$ is a \emph{\CuMor{}} if it preserves the order, suprema of increasing sequences, and the way-below relation.
A \emph{generalized \CuMor} is a monoid map that preserves order and suprema of increasing sequences (but not necessarily the way-below relation).

The following properties, which will often be considered throughout the paper, are also satisfied in the Cuntz semigroup of any \ca{};
see \cite[Proposition~4.6]{AntPerThi18TensorProdCu} and its precursor \cite[Lemma~7.1]{RorWin10ZRevisited} for \axiomO{5}, \cite[Proposition~5.1.1]{Rob13Cone} for \axiomO{6}, and \cite[Proposition~2.2]{AntPerRobThi21Edwards} for \axiomO{7}.
\begin{itemize}
\item[\axiomO{5}]
For every $x, y, x', y', z\in S$ satisfying $x+y\leq z$ and $x'\ll x$ and $y'\ll y$, there exists $c \in S$ such that $y'\ll c$ and $x'+c\leq z \leq x+c$.

This property is often applied with $y'=y=0$.
\item[\axiomO{6}]
For every $x, x', y, z\in S$ satisfying $x'\ll x\ll y+z$, there exist $v, w \in S$ such that
\[
v \leq x,y, \quad 
w \leq x,z, \andSep
x'\leq v+w.
\]
\item[\axiomO{7}]
For every $x, x', y, y', w \in S$ satisfying $x'\ll x\leq w$ and $y'\ll y\leq w$, there exists $z \in S$ such that $x',y'\ll z\leq w,x+y$.
\end{itemize}
\end{pgr}

Given an element $x$ in a \CuSgp{}, we denote by $\infty x$ the supremum of the increasing sequence $(nx)_n$.

\begin{pgr}[The Global Glimm Property and nowhere scatteredness]
\label{pgr:GGPNSCa}
A \ca{} $A$ is said to be \emph{nowhere scattered} if no hereditary sub-\ca{} of $A$ has a nonzero one-dimensional representation. 
Equivalently, $A$ is nowhere scattered if and only if $A$ has no nonzero elementary ideal-quotients; see \cite[Definition~A]{ThiVil21arX:NowhereScattered} and \cite[Theorem~3.1]{ThiVil21arX:NowhereScattered}.

We say that $A$ has the \emph{Global Glimm Property} (in the sense of \cite[Definition~4.12]{KirRor02InfNonSimpleCalgAbsOInfty}) if, for every $a\in A_+$ and $\varepsilon >0$, there exists a square-zero element $r\in\overline{aAa}$ such that $(a-\varepsilon )_+\in\overline{\rm span}ArA$; see \cite[Section~3]{ThiVil23Glimm}.

A \ca{} satisfying the Global Glimm Property is always nowhere scattered. 
The converse remains open, and is known as the \emph{Global Glimm Problem}. 
The problem has been answered affirmatively under the additional assumption of real rank zero (\cite{EllRor06Perturb}) or stable rank one (\cite{AntPerRobThi22CuntzSR1}).

A \CuSgp{} is said to be $(2,\omega )$-divisible if, for every pair $x',x\in S$ with $x'\ll x$, there exists $y\in S$ such that $2y\leq x$ and $x'\leq \infty y$; see \cite[Definition~5.1]{RobRor13Divisibility}.

For a detailed study of the Global Glimm Problem and its relation with the Cuntz semigroup we refer to \cite{ThiVil23Glimm}; 
see also \cite{Vil23pre:MultNSCa}. 
Among other results, it follows from \cite[Theorem~3.6]{ThiVil23Glimm} that a \ca{} $A$ has the Global Glimm Property if and only if $\Cu (A)$ is $(2,\omega )$-divisible.
\end{pgr}

\section{Soft operators and strongly soft Cuntz classes}
\label{sec:Soft}

In this section, we first recall the definitions of (completely) soft operators in \ca{s} and of strongly soft elements in \CuSgp{s}.
We then connect these notions and show that, for a \ca{} $A$ with the Global Glimm Property, an element in the Cuntz semigroup $\Cu(A)$ is strongly soft if and only if it has a soft representative;
see \autoref{prp:SoftStSoftEq} and \autoref{prp:CharSoftCuClassGGP}.

\medskip

As defined in \cite[Definition~4.2]{ThiVil23arX:Soft}, an element $x$ in a \CuSgp{} $S$ is \emph{strongly soft} if for all $x' \in S$ with $x' \ll x$ there exists $t\in S$ such that
\[
x'+t \ll x, \andSep 
x'\ll \infty t.
\]

This notion of softness is stronger than the one considered in \cite[Definition~5.3.1]{AntPerThi18TensorProdCu}. 
However, if $S$ is residually stably finite, both notions agree; 
see \cite[Proposition~4.6]{ThiVil23arX:Soft}. 
In particular, this applies to weakly cancellative \CuSgp{s} (see \autoref{pgr:StR1} below).

\smallskip
 
As mentioned in the introduction, a positive element $a$ in a \ca{} $A$ is said to be \emph{soft} if its hereditary sub-\ca{} has no nonzero unital quotients. 
This definition can be seen as a generalization of \emph{pure positivity}, a notion introduced in \cite[Definition~2.1]{PerTom07Recasting} for simple \ca{s}.
An element $a \in A_+$ is said to be \emph{completely soft} if $(a-\varepsilon)_+$ is soft for every $\varepsilon>0$, where $(a-\varepsilon)_+$ denotes the `cut-down' of $a$ given by applying functional calculus to $a$ with the function $f(t)=\max\{t-\varepsilon,0\}$.

As in \cite[Definition~5.2]{ThiVil23arX:Soft}, we say that a \ca{} $A$ has an \emph{abundance of soft elements} if, for every $a\in A_+$ and $\varepsilon >0$, there exists a positive, soft element $b \in \overline{aAa}$ such that $(a-\varepsilon )_+\in\overline{\rm span}AbA$. 
By \cite[Proposition~7.7]{ThiVil23arX:Soft}, any \ca{} with the Global Glimm Property has an abundance of soft elements.
 
If $a \in A_+$ is soft, then its Cuntz class $[a]$ is strongly soft; 
see \cite[Proposition~4.16]{ThiVil23arX:Soft}.
Conversely, we prove in \autoref{prp:SoftStSoftEq} below that if $A$ has an abundance of soft elements (in particular, if $A$ has the Global Glimm Property), then every strongly soft Cuntz class arises this way, that is, a Cuntz class $[b] \in \Cu(A)$ is strongly soft if and only if there exists a soft element $a \in (A\otimes\KK)_+$ with $b\sim a$.
It remains unclear if this also holds for general \ca{s};
see \cite[Question~4.17]{ThiVil23arX:Soft}.

Given $a,b\in A_+$, we will write $a\vartriangleleft b$ whenever $a\in\overline{\rm span}AbA$.
We say that two positive operators $a$ and $b$ in a \ca{} are \emph{orthognal} if $ab=0$.

The next result is the \ca{ic} analog of \cite[Theorem~4.14(2)]{ThiVil23arX:Soft}.

\begin{prp}
\label{prp:AbsorbingSoftCAlg}
Let $a$ and $b$ be orthogonal positive elements in a \ca{} such that $a \lhd b$, and such that $b$ is soft.
Then $a+b$ is soft. 
\end{prp}
\begin{proof}
By \cite[Proposition~3.6]{ThiVil23arX:Soft}, a positive element $c$ in a \ca{} is soft if and only if for every $\varepsilon>0$ there exists $r \in (\overline{cAc})_+$ such that $r$ is orthogonal to $(c-\varepsilon)_+$ and such that $c \lhd r$.
Using this characterization for $b$, we show that it is satisfied for $a+b$.

To verify that $a+b$ is soft, let $\varepsilon>0$.
Using that $b$ is soft, we obtain $r \in (\overline{bAb})_+$ such that $r$ is orthogonal to $(b-\varepsilon)_+$ and such that $b \lhd r$.
Since $a$ and $b$ are orthogonal, we have
\[
((a+b)-\varepsilon)_+
= (a-\varepsilon)_+ + (b-\varepsilon)_+.
\]

Since $r$ belongs to $\overline{bAb}$, it is also orthogonal to $a$, and thus also orthogonal to $((a+b)-\varepsilon)_+$.
Further, we have $a+b \lhd b \lhd r$, as desired.
\end{proof}

\begin{lma}
\label{prp:SoftSupDense}
Let $A$ be a \ca{} with an abundance of soft elements, let $a \in A_+$ be such that $x:=[a]\in\Cu (A)$ is strongly soft, and let $x' \in \Cu(A)$ satisfy $x'\ll x$.
Then there exists a positive, completely soft element $b \in \overline{aAa}$ such that
\[
x' \ll [b] \ll x.
\]
\end{lma}
\begin{proof}
Choose $x''\in \Cu (A)$ such that $x'\ll x''\ll x$. 
Using that $x$ is strongly soft, we know that there exists $t\in\Cu (A)$ such that $x''\ll \infty t$ and $x''+t\ll x$.
Choose orthogonal positive elements $c,d \in A\otimes\KK$ and $\varepsilon>0$ such that
\[
x'' = [c], \quad
t = [d], \quad
x' \ll [(c-\varepsilon)_+], \andSep
x'' \ll \infty[(d-\varepsilon)_+].
\]

Using that $c+d \precsim a$, we can apply R{\o}rdam's lemma (see, for example, \cite[Theorem~2.30]{Thi17:CuLectureNotes}) to obtain $x \in A\otimes\KK$ such that
\[
((c+d)-\varepsilon)_+ = xx^*, \andSep
x^*x \in \overline{aAa}.
\]

Set
\[
c' := x^*(c-\varepsilon)_+x, \andSep
d' := x^*(d-\varepsilon)_+x.
\]
Then $c',d' \in \overline{aAa}$.
Since $c$ and $d$ are orthogonal, we have
\[
((c+d)-\varepsilon)_+
= (c-\varepsilon)_+ + (d-\varepsilon)_+.
\]
It follows that $c'$ and $d'$ are orthogonal, and that $c' \sim (c-\varepsilon)_+$ and $d' \sim (d-\varepsilon)_+$.

In particular, we have $x'' \ll \infty[(d-\varepsilon)_+] = \infty[d']$, and we obtain $\delta>0$ such that $x'' \ll \infty[(d'-\delta)_+]$.
Applying that $A$ has an abundance of soft elements for~$d'$ and~$\delta$, we obtain a soft element $e \in (\overline{d'Ad'})_+$ such that $(d'-\delta)_+ \lhd e$.
Since~$c'$ and~$d'$ are orthogonal, and $e$ belongs to $\overline{d'Ad'}$, it follows that $c'$ and $e$ are orthogonal.

Using that positive elements $g,h$ in a \ca{} satisfy $g \lhd h$ if and only if $[g] \leq \infty[h]$, we have
\[
[c'] 
= [(c-\varepsilon)_+]
\leq [c]
= x''
\leq \infty[(d'-\delta)_+]
\leq \infty[e]
\]
and thus $c' \lhd e$.
By \autoref{prp:AbsorbingSoftCAlg}, $c'+e$ is soft.

Note that $c'$ and $e$ belong to $\overline{aAa}$.
In particular, $c'+e$ belongs to $A_+$, and we can apply \cite[Theorem 6.9]{ThiVil23arX:Soft} to obtain a completely soft element $f \in A_+$ such that $\overline{fAf} = \overline{(c'+e)A(c'+e)} \subseteq \overline{aAa}$.
Then $f \in \overline{aAa}$, and therefore $[f] \leq [a] = x$.
Further, we have
\[
x' \ll [(c-\varepsilon)_+] = [c'] \leq [c'+e] = [f].
\]

Choose $\delta>0$ such that
\[
x' \ll [(f-\delta)_+],
\]
and set $b := (f-\delta)_+$.
Since cut-downs of $(f-\delta)_+$ are also cut-downs of $f$, we see that~$b$ is completely soft.
Further, we have
\[
x' \ll [b] = [(f-\delta)_+] \ll [f] \leq x,
\]
which shows that $b$ has the desired properties.
\end{proof}

A unital \ca{} is said to have \emph{stable rank one} if its invertible elements are norm-dense;
and a general \ca{} is said to have stable rank one if its minimal unitization does;
see \cite[Section~V.3.1]{Bla06OpAlgs}.

A \ca{} is said to have \emph{weak stable rank one} if $A\subseteq \overline{{\rm Gl}(\tilde{A})}$. 
Any stable \ca{} has weak stable rank one; 
see \cite[Lemma~4.3.2]{BlaRobTikTomWin12AlgRC}.

\begin{thm}
\label{prp:SoftStSoftEq}
Let $A$ be a \ca{} with an abundance of soft elements, and let $a\in A_+$ be such that $[a]\in\Cu (A)$ is strongly soft. 
Then there exists a sequence $(a_n)_n$ of completely soft elements in $(\overline{aAa})_+$ such that $([a_n])_n$ in $\Cu(A)$ is $\ll$-increasing with $[a] = \sup_n [a_n]$.

If, moreover, $A$ has weak stable rank one, then $[a]$ is strongly soft if and only if there exists a completely soft element $b\in A_+$ such that $[a]=[b]$.
\end{thm}
\begin{proof}
Choose a $\ll$-increasing sequence $(x_n)_n$ in $\Cu (A)$ with supremum $[a]$.
We will inductively choose completely soft elements $a_n \in (\overline{aAa})_+$ such that
\[
x_n \ll [a_n] \ll [a], \andSep
[a_n] \ll [a_{n+1}]
\]
for $n\in\NN$.
To start, apply \autoref{prp:SoftSupDense} for $x_0 \ll [a]$ to obtain a completely soft element $a_0 \in (\overline{aAa})_+$ such that $x_0 \ll [a_0] \ll [a]$.
Assuming we have chosen $a_0,\ldots,a_n$, find $x_n' \in \Cu(A)$ such that $[a_n],x_n \ll x_n' \ll [a]$.
Applying \autoref{prp:SoftSupDense} for $x_n' \ll [a]$ we obtain a completely soft element $a_{n+1} \in (\overline{aAa})_+$ such that $x_n' \ll [a_{n+1}] \ll [a]$.
Proceeding inductively, we obtain the desired sequence $(a_n)_n$.

\smallskip

Next, assume that $A$ has weak stable rank one.
By \cite[Proposition~4.16]{ThiVil23arX:Soft}, soft operators have strongly soft Cuntz classes.
Conversely, assuming that $[a]$ is strongly soft, we will show that $[a] = [b]$ for some completely soft element $b\in A_+$.

Let $(a_n)_n$ be as above.
We will show that $\sup_n [a_n]$ (which is $[a]$), has a soft representative. 
Given $c,d \in A_+$ we will write $c \sim_u d$ if there exists a unitary $u\in\tilde{A}$ such that $c=udu^*$; 
and we write $c \subseteq d$ if $\overline{cAc} \subseteq \overline{dAd}$.
 
Using \cite[{\S}2.5]{Thi17:CuLectureNotes}, one can find a sequence $(\delta_n)_n$ in $(0,\infty)$, and a sequence of contractive elements $(b_n)_n$ in $A_+$ such that 
\[
\begin{matrix}
a_1 & \precsim & a_2 & \precsim & a_3 & \precsim & \ldots\\
\rotatebox{90}{$\leq$} & & \rotatebox{90}{$\leq$} & & \rotatebox{90}{$\leq$} & \\
(a_1-\delta_1)_+ & & (a_2-\delta_2)_+ & & (a_3-\delta_3)_+ & & \ldots \\
\rotatebox{90}{$\sim_u$} & & \rotatebox{90}{$\sim_u$} & & \rotatebox{90}{$\sim_u$} & \\
b_1 & \subseteq & b_2 & \subseteq & b_3 & \subseteq & \ldots 
\end{matrix} 
\]
and, setting $b_\infty:=\sum_n \frac{1}{2^n \Vert b_n \Vert } b_n$, such that $[b_\infty]=\sup_n [a_n]$.

For each $n\in\NN$, since $a_n$ is completely soft, so is the element $(a_n-\delta_n )_+$.
Since $(a_n-\delta_n )_+$ and~$b_n$ are unitarily equivalent, they generate $\ast$-isomorphic hereditary sub-\ca{s} of $A$, and it follows that $b_n$ is completely soft as well.

Further, since $b_0\subseteq b_1\subseteq\ldots$ and $b_\infty=\sum_n \frac{1}{2^n \Vert b_n \Vert } b_n$, the sequence of hereditary sub-\ca{s} $\overline{b_{n}Ab_{n}}$ is increasing with $\overline{b_\infty A b_\infty} = \overline{\bigcup_n \overline{b_n A b_n}}$.
Since each $\overline{b_n A b_n}$ has no nonzero unital quotients, it follows from \cite[Proposition~2.17]{ThiVil23arX:Soft} that neither does $\overline{b_\infty A b_\infty}$. 
This proves that $b_\infty$ is soft.

Note that $b_\infty$ belongs to $A_+$.
Applying \cite[Theorem 6.9]{ThiVil23arX:Soft}, we obtain a completely soft element $b \in A_+$ such that $\overline{bAb} = \overline{b_\infty A b_\infty}$.
Then $[b]=[b_\infty]=[a]$, as desired.
\end{proof}

\begin{cor}
\label{prp:CharSoftCuClassGGP}
Let $A$ be a \ca{} with the Global Glimm Property, and let $x \in \Cu(A)$.
Then $x$ is strongly soft if and only if there exists a soft element $a \in (A\otimes\KK)_+$ with $x = [a]$.
\end{cor}
\begin{proof}
It follows from \cite[Theorem~3.6]{ThiVil23Glimm} that $A\otimes\KK$ has the Global Glimm Property.
Hence, $A\otimes\KK$ has an abundance of soft elements by \cite[Proposition~7.7]{ThiVil23Glimm}.
Further, $A\otimes\KK$ has weak stable rank one by \cite[Lemma~4.3.2]{BlaRobTikTomWin12AlgRC}.
Now the result follows from \autoref{prp:SoftStSoftEq}.
\end{proof}

\begin{pgr}[The strongly soft subsemigroup]
\label{pgr:SoftSubCu} 
Given a \CuSgp{} $S$, we let $S_\soft$ denote the set of strongly soft elements in $S$. 
By \autoref{prp:CharSoftCuClassGGP}, given a \ca{} $A$ with the Global Glimm Property, we have
\[
\Cu(A)_\soft = \big\{ [a] : a \in (A\otimes\KK)_+ \text{ soft} \big\}.
\]

In particular, if $A$ is stably finite, simple, and unital, it follows from \cite[Proposition~4.16]{ThiVil23arX:Soft} that the subset $\Cu (A)_\soft\setminus\{ 0\}$ coincides with $\Cu_{+} (A)$, the set of Cuntz classes of purely positive elements as introduced in \cite[Definition~2.1]{PerTom07Recasting}, see also \cite[Definition~3.8]{AsaGolPhi21RadCompCrProd}.
\end{pgr}

Given a \CuSgp{} $S$, a \emph{sub-\CuSgp{}} in the sense of \cite[Definition~4.1]{ThiVil21DimCu2} is a submonoid $T \subseteq S$ that is a \CuSgp{} for the inherited order, and such that the inclusion map $T \to S$ is a \CuMor{}. 

\begin{prp}
\label{prp:SoftPartAxioms}
Let $S$ be a $(2,\omega)$-divisible \CuSgp{} that satisfies \axiomO{5}.
Then, $S_\soft$ is a sub-\CuSgp{} that also satisfies \axiomO{5}.

If $S$ also satisfies \axiomO{6} (respectively \axiomO{7}), then so does $S_\soft$.
\end{prp}
\begin{proof}
By \cite[Proposition~7.7]{ThiVil23arX:Soft}, if a \CuSgp{} is $(2,\omega)$-divisible and satisfies \axiomO{5}, then it has an abundance of soft elements, which then by \cite[Proposition~5.6]{ThiVil23arX:Soft} implies that its strongly soft elements form a sub-\CuSgp{}.
Thus, $S_\soft$ is a sub-\CuSgp{}.

Let us verify that $S_\soft$ satisfies \axiomO{5}.
By \cite[Theorem~4.4(1)]{AntPerThi18TensorProdCu} it suffices to show that for all $x',x,y',y,z',z \in S_\soft$ satisfying
\begin{equation}
\label{eq:SoftPartAxioms:O5Assumption}
x' \ll x, \quad
y' \ll y, \andSep
x+y \ll z' \ll z,
\end{equation}
there exist $c',c \in S_\soft$ such that
\begin{equation}
\label{eq:SoftPartAxioms:O5Conclusion}
x'+c \ll z, \quad
z' \ll x+c', \andSep 
y' \ll c' \ll c.
\end{equation}

So let $x',x,y',y,z',z \in S_\soft$ satisfy \eqref{eq:SoftPartAxioms:O5Assumption}.
Choose $v',v \in S_\soft$ such that $z' \ll v' \ll v \ll z$.
Applying \axiomO{5}, we obtain $b \in S$ such that
\[
x'+b \leq v' \leq x+b, \andSep
y' \ll b.
\]

Using that $v' \ll v$ and that $v$ is strongly soft, we apply \cite[Proposition~4.13]{ThiVil23arX:Soft} to find $t \in S_\soft$ such that $v'+t \leq v \leq \infty t$.
Set $c := b+t$.
Since $b \leq v' \leq v \leq \infty t$ and $t$ is strongly soft, we have $c \in S_\soft$ by \cite[Theorem~4.14(2)]{ThiVil23arX:Soft}.
Thus, one gets
\[
x'+c
= x'+b+t
\leq v'+t
\leq v
\ll z,
\]
and
\[
z' \ll v' \leq x+b \leq x+c, \andSep
y' \ll b \leq c.
\]

Using also that $S_\soft$ is a \CuSgp{} and $c \in S_\soft$, we can find $c' \in S_\soft$ such that
\[
c' \ll c, \quad
z' \ll x+c', \andSep
y' \ll c'.
\]

This shows that $c'$ and $c$ satisfy \eqref{eq:SoftPartAxioms:O5Conclusion}, as desired.

That $S_\soft$ satisfies \axiomO{6} (respectively \axiomO{7}) whenever $S$ does is proven analoguously.
\end{proof}

\section{Separative \texorpdfstring{\CuSgp{s}}{Cu-semigroups}}
\label{sec:SepCu}

We introduce in \autoref{dfn:LefSofSep} the notion of left-soft separativity, a weakening of weak cancellation (\autoref{pgr:StR1}) that is satisfied in the Cuntz semigroup of every \ca{} with stable rank one or strict comparison of positive elements; 
see \autoref{prp:WCimpLSSep} and \autoref{prp:StCompLSSep} respectively. 
We also prove in \autoref{prp:PerfCond} that, among strongly soft elements, the notions of unperforation and almost unperforation coincide.

\begin{pgr}[Cuntz semigroups of stable rank one \ca{s}]
\label{pgr:StR1}
Let $A$ be a stable rank one \ca{}. 
As shown in \cite[Theorem~4.3]{RorWin10ZRevisited}, the Cuntz semigroup $\Cu (A)$ satisfies a cancellation property termed \emph{weak cancellation}: 
If $x,y,z \in \Cu(A)$ satisfy $x+z\ll y+z$, then $x\ll y$.

If $A$ is also separable, then $\Cu(A)$ is \emph{inf-semilattice ordered}, that is, for every pair of elements $x,y \in \Cu(A)$ their infimum $x \wedge y$ exists, and for every $x,y,z \in \Cu(A)$ one has $(x+z) \wedge (y+z) = (x \wedge y) +z$; 
see \cite[Theorem~3.8]{AntPerRobThi22CuntzSR1}.
\end{pgr}

As defined in \cite{ThiVil21arX:ZeroDimCu}, a \CuSgp{} is \emph{separative} if $x \ll y$ whenever $x+t\ll y+t$ with $t\ll\infty x,\infty y$. 
This and other cancellation properties will be studied in more detail in \cite{ThiVil21arX:ZeroDimCu}.

For the results in this paper, we will need the following tailored definition:

\begin{dfn}
\label{dfn:LefSofSep}
We say that a \CuSgp{} $S$ is \emph{left-soft separative} if, for any triple of elements $y,t\in S$ and $x\in S_\soft$ satisfying 
\[
x+t \ll y+t,\quad 
t \ll \infty x,\andSep 
t \ll \infty y,
\]
we have $x\ll y$.
\end{dfn}

\begin{prp}
\label{prp:WCimpLSSep}
Every weakly cancellative \CuSgp{} is separative, and every separative \CuSgp{} is left-soft separative.

In particular, the Cuntz semigroup of every stable rank one \ca{} is left-soft separative.
\end{prp}
\begin{proof}
It follows directly from the definitions that weak cancellation is stronger than left-soft separativitiy.
By \cite[Theorem~4.3]{RorWin10ZRevisited}, the Cuntz semigroup of a stable rank one \ca{} is weakly cancellative.
\end{proof}

\begin{lma}
\label{rmk:LSSepEqu}
Let $S$ be a $(2,\omega )$-divisible \CuSgp{} satisfying \axiomO{5}. 
Then, $S$ is left-soft separative if and only if for all $y,t',t\in S$ and $x\in S_\soft$ satisfying 
\[
x+t\leq y+t',\quad 
t'\ll t,\quad 
t'\ll \infty y,\andSep 
t'\ll \infty x,
\]
we have $x \leq y$.
\end{lma}
\begin{proof}
The backwards implication is straightforward to verify and even holds for general \CuSgp{s}.
To show the forward implication, assume that $S$ is left-soft separative, and let $x,y,t',t\in S$ as in the statement. 
By \autoref{prp:SoftPartAxioms}, we know that $S_\soft$ is a sub-\CuSgp{}. 
In particular, $x$ can be written as the supremum of a $\ll$-increasing sequence of strongly soft elements.

Take $x'\in S_\soft$ such that $x'\ll x$. 
We have 
\[
x'+t'\ll x+t\leq y+t',\quad 
t' \ll \infty x,\andSep 
t' \ll \infty y.
\]

By left-soft separativity, we deduce $x'\ll y$. 
Since $x$ is the supremum of such $x'$, one gets $x\leq y$, as required.
\end{proof}

\begin{lma}
\label{prp:CancelSoft}
Let $S$ be a left-soft separative, $(2,\omega)$-divisible \CuSgp{} satisfying \axiomO{5}, and let $x,t \in S_\soft$ and $y,t' \in S$ satisfy 
\[
x+t \leq y+t', \quad 
t' \ll t, \quad 
t' \ll \infty y.
\]

Then $x \leq y$.
\end{lma}
\begin{proof}
Take $t''\in S$ such that $t'\ll t''\ll t$. 
Using that $t$ is strongly soft, one finds $s\in S_\soft$ such that $t''+s\leq t\leq \infty s$; 
see \cite[Proposition~4.13]{ThiVil23arX:Soft}. 
Note that, since $x$ and $s$ are strongly soft, so is $x+s$ by \cite[Theorem~4.14]{ThiVil23arX:Soft}. 
We get
\[
(x+s)+t'' = x+(s+t'') \leq x+t \leq y+t'.
\]
Further, we have $t'\ll \infty y$ and $t'\ll t''\leq \infty s\leq \infty (x+s)$.

An application of \autoref{rmk:LSSepEqu} shows that $x+s\leq y$ and, therefore, that $x\leq y$.
\end{proof}

The following result shows that three different versions of unperforation coincide for the semigroup of strongly soft elements in a \CuSgp.
Given elements~$x$ and $y$ in a partially ordered monoid, one writes $x<_sy$ if there exists $n \geq 1$ such that $(n+1)x \leq ny$;
and one writes $x \leq_p y$ if there exists $n_0 \in \NN$ such that $nx \leq ny$ for all $n \geq n_0$.
We refer to \cite[Chapter~5]{AntPerThi18TensorProdCu} for details regarding these definitions.

\begin{prp}
\label{prp:PerfCond}
Let $S$ be a \CuSgp{}. 
The following are equivalent:
\begin{itemize}
\item[(1)]
$S_\soft$ is unperforated: 
If $x,y \in S_\soft$ and $n \geq 1$ satisfy $nx \leq ny$, then $x \leq y$.
\item[(2)] 
$S_\soft$ is nearly unperforated: 
If $x,y \in S_\soft$ satisfy $x \leq_p y$, then $x \leq y$.
\item[(3)] 
$S_\soft$ is almost unperforated:
If $x,y \in S_\soft$ satisfy $x <_s y$, then $x \leq y$.
\end{itemize}
\end{prp}
\begin{proof}
In general, (1)~implies~(2), which implies~(3);
see \cite[Proposition~5.6.3]{AntPerThi18TensorProdCu}.
To verify that~(3) implies~(1), let $x,y\in S_\soft$ and $n\geq 1$ satisfy $nx\leq ny$.
Then $\widehat{x} \leq \widehat{y}$;
see \autoref{pgr:FuncRank}.
By \cite[Proposition~4.5]{ThiVil23arX:Soft}, $x$ is functionally soft.
Thus, we deduce from \cite[Theorem~5.3.12]{AntPerThi18TensorProdCu} that $x \leq y$, as desired.
\end{proof}

\begin{lma}
\label{prp:AUnpLSSep}
Every almost unperforated \CuSgp{} satisfying \axiomO{5} is left-soft separative.
\end{lma}
\begin{proof}
Let $S$ be an almost unperforated \CuSgp{} satisfying \axiomO{5}. 
To verify that~$S$ is left-soft separative, let $y,t\in S$ and $x\in S_\soft$ satisfy $x+t\ll y+t$ and $t\ll\infty x,\infty y$. 
Choose $y'\in S$ such that
\[
x+t\ll y'+t, \quad t\ll\infty y', \andSep y'\ll y.
\]

Then $x \leq_p y'$ by \cite[Proposition~5.6.8(ii)]{AntPerThi18TensorProdCu}.
In particular, there exists $k\in\NN$ such that $kx\leq ky'$, and thus $\widehat{x} \leq \widehat{y'}$;
see \autoref{pgr:FuncRank}.
By \cite[Proposition~4.5]{ThiVil23arX:Soft}, $x$ is functionally soft.
Using that $S$ is almost unperforated, we obtain that $x \leq y' \ll y$, by \cite[Theorem~5.3.12]{AntPerThi18TensorProdCu}.
\end{proof}

A \ca{} $A$ is said to have \emph{strict comparison of positive elements} if, for all $a,b \in (A\otimes\KK)_+$ and $\varepsilon>0$, one has that $d_\tau(a) \leq (1-\varepsilon)d_\tau(b)$ for all $\tau$ implies $a \precsim b$.

\begin{prp}
\label{prp:StCompLSSep}
Let $A$ be a \ca{} with strict comparison of positive elements. 
Then $\Cu (A)$ is left-soft separative.
\end{prp}
\begin{proof}
A \ca{} has strict comparison of positive elements if and only if its Cuntz semigroup is almost unperforated; 
see \cite[Proposition~6.2]{EllRobSan11Cone}. 
Since every Cuntz semigroup satisfies \axiomO{5}, the result follows from \autoref{prp:AUnpLSSep}.
\end{proof}

Since every $\mathcal{Z}$-stable \ca{} has strict comparison of positive elements (see \cite[Theorem~4.5]{Ror04StableRealRankZ}), one gets the following:

\begin{cor}
The Cuntz semigroup of every $\mathcal{Z}$-stable \ca{} is left-soft separative.
\end{cor}

\section{Ranks and soft elements}
\label{sec:SoftRanks}

Given a $(2,\omega)$-divisible \CuSgp{} $S$ satisfying \axiomO{5}-\axiomO{7} (for example, the Cuntz semigroup of a \ca{} with the Global Glimm Property) and an element $x\in S$, we show in \autoref{prp:SoftRankS} that there exists a strongly soft element~$w$ below~$x$ which agrees with $x$ at the level of functionals, that is, the rank of $x$ coincides with the rank of $w$; see \autoref{pgr:FuncRank}. 
Paired with \autoref{prp:SoftStSoftEq}, this implies that the rank of any positive element in a \ca{} satisfying the Global Glimm Property is the rank of a soft element (\autoref{prp:SoftRankCuGGP}).

Using \autoref{prp:SoftRankS}, we also prove that $F(S)$, the set of functionals on $S$, is homeomorphic to $F(S_\soft )$; 
see \autoref{prp:FScongFSSoft}.

\begin{pgr}[Functionals and ranks]
\label{pgr:FuncRank}
Given a \CuSgp{} $S$, we will denote by $F(S)$ the set of its \emph{functionals}, that is to say, the set of monoid morphisms $S\to [0,\infty ]$ that preserve the order and suprema of increasing sequences.
If $S$ satisfies \axiomO{5}, then $F(S)$ becomes a compact, Hausdorff space -- and even an algebraically ordered compact cone \cite[Section~3]{AntPerRobThi21Edwards} -- when equipped with a natural topology 
\cite{EllRobSan11Cone, Rob13Cone, Kei17CuSgpDomainThy}.

Given a \ca{}, the cone $\QT(A)$ of lower-semicontinuous 2-quasitraces on $A$ is naturally isomrphic to $F(\Cu(A))$, as shown in \cite[Theorem~4.4]{EllRobSan11Cone}.

We let $\LAff(F(S))$ denote the monoid of lower-semicontinuous, affine functions $F(S)\to(-\infty,\infty]$, equipped with pointwise order and addition.
For $x\in S$, the \emph{rank} of $x$ is defined as the map $\widehat{x} \colon F(S) \to [0,\infty]$ given by
\[
\widehat{x}(\lambda) := \lambda(x)
\]
for $\lambda \in F(S)$.
The function $\widehat{x}$ belongs to $\LAff(F(S))$ and the \emph{rank problem} of determining which functions in $\LAff(F(S))$ arise this way has been studied extensively in \cite{Thi20RksOps} and \cite{AntPerRobThi22CuntzSR1}.

Sending an element $x \in S$ to its rank $\widehat{x}$ defines a monoid morphism from $S$ to $\LAff (F(S))$ which preserves both the order and suprema of increasing sequences.
\end{pgr}

\begin{lma}
\label{prp:O5Appl}
Let $S$ be a $(2,\omega )$-divisible \CuSgp{} satisfying \axiomO{5}, and let $u\in S_\soft$ and $u',x\in S$ be such that
\[
u'\ll u\ll x.
\]
 
Then, there exists $c\in S_\soft $ satisfying 
\[
u'+2c\leq x\leq \infty c.
\]
\end{lma}
\begin{proof}
Let $u''\in S$ be such that $u'\ll u''\ll u$. 
By \cite[Proposition~4.13]{ThiVil23arX:Soft}, there exists $s\in S$ satisfying 
\[
u''+s\leq u\leq \infty s.
\]

Since $u''\ll u\leq \infty s$, there exists $s' \in S$ such that
\[
s' \ll s, \andSep 
u''\ll \infty s'.
\]

We have 
\[
u''+s \leq x. \quad
u' \ll u'', \andSep
s'\ll s.
\]
Applying \axiomO{5}, we obtain $d\in S$ such that $u'+d\leq x\leq u''+d$ with $s'\leq d$. 
Since $u''\leq \infty s'$, it follows that $x\leq \infty d$.
Finally, apply \cite[Proposition~7.7]{ThiVil23arX:Soft} to $d$ in order to obtain $c\in S_\soft$ such that $2c\leq d\leq \infty c$. 
This element satisfies the required conditions.
\end{proof}

A \CuSgp{} $S$ is said to be \emph{countably based} if it contains a countable subset $D\subseteq S$ such that every element in $S$ can be written as the supremum of an increasing sequence of elements in $D$. 
Separable \ca{s} have countably based Cuntz semigroups; see, for example, \cite{AntPerSan11PullbacksCu}.

\begin{lma}
\label{prp:PreLxDirected}
Let $S$ be a countably based, $(2,\omega )$-divisible \CuSgp{} satisfying \axiomO{5}-\axiomO{7}, and let $x\in S$.
Consider the set
\[
L_x :=  \big\{ u'\in S : u'\ll u\ll x \text{ for some } u\in S_\soft\big\}.
\]

Then, for every $k \in \NN$, $x'\in S$ such that $x'\ll x$, and $u',v'\in L_x$, there exists a strongly soft element $w'\in L_x$ such that 
\[
u'\ll w',\quad 
x'\ll \infty w',\andSep 
\frac{k}{k+1}\widehat{v'}\leq \widehat{w'} \text{ in $\LAff (F(S))$}.
\]

If, additionally, $S$ is left-soft separative, $w'$ may be chosen so that $v' \ll w'$.
\end{lma}
\begin{proof}
Let $u',v'\in L_x$, let $x' \in S$ satisfy $x' \ll x$, and let $k \in \NN$.
By definition, there exist $u,v\in S_\soft$ such that 
\[
u'\ll u\ll x,\andSep v'\ll v\ll x.
\]
Choose $y',y\in S$ such that
\[
x' \ll y' \ll y \ll x, \quad
v \ll y', \andSep
u \ll y'.
\]
Using that $S_\soft$ is a sub-\CuSgp{} by \autoref{prp:SoftPartAxioms}, we can choose elements $u'',u''',v'' \in S_\soft$ such that 
\[
u' \ll u'' \ll u''' \ll u, \andSep
v' \ll v'' \ll v.
\]
 
Applying \autoref{prp:O5Appl} for $u'''\ll u\ll y$ and $v''\ll v\ll y$, we obtain $c,d \in S_\soft$ such that 
\[
u'''+c\leq y\leq \infty c,\andSep 
v''+2d\leq y\leq \infty d.
\]

Then, applying \cite[Proposition~4.10]{ThiVil23Glimm} for $y'\ll y\leq\infty c,\infty d$, we get $e\in S$ such that
\[
y' \ll \infty e, \andSep e\ll c,d.
\]

By \cite[Proposition~7.7]{ThiVil23arX:Soft}, there exists a strongly soft element $e_0$ such that $e_0 \leq e \leq \infty e_0$.
Replacing $e$ by $e_0$, we may assume that $e \in S_\soft$.
Using again that $S_\soft$ is a sub-\CuSgp{}, we can find $e',e''\in S_\soft$ satisfying 
\[
y' \ll \infty e', \andSep
e'\ll e''\ll e.
\]

By \cite[Proposition~4.13]{ThiVil23arX:Soft}, there exists $r\in S$ such that
\[
e''+r\leq e\leq \infty r.
\]
Since $e''\ll e$, we can find $r' \in S$ such that
\[
r' \ll r, \andSep 
e''\leq \infty r'.
\] 

Thus, one has 
\[
e''+(r+u''')\leq e+u'''\leq c+u'''\leq y, \quad
e' \ll e'', \andSep
r'+u'' \ll r+u'''.
\]
Applying \axiomO{5}, we obtain $z\in S$ such that
\[
e'+z \leq y \leq e''+z, \andSep 
r'+u''\leq z.
\]

Using again that $S_\soft$ is a sub-\CuSgp{}, choose $d' \in S_\soft$ such that
\[
e \ll d' \ll d.
\]
We have
\begin{equation}
\label{eq:Separative}
(v''+d)+d=v''+2d\leq y \leq z+e'' \leq z+d',
\end{equation}
with $v''+d\in S_\soft$. 
Note that 
\[
d'\ll d\leq \infty (v''+d),\andSep 
d'\ll d\leq y\leq z+e''\leq z+\infty r'\leq \infty z.
\]

In particular, since $d'\ll \infty z$, there exists $M\in\NN$ such that $d'\leq Mz$.
Set 
\[
l:=  \infty (u''+v''), \andSep
w:=e'+ (z\wedge l),
\]
where $z\wedge l$ exists because $l$ is idempotent, and $S$ is countably based and satisfies \axiomO{7}; see \cite[Theorem~2.4]{AntPerRobThi21Edwards}. 

Note that, since $l\leq \infty y'\leq  \infty e'$ and $e'\in S_\soft$, it follows from \cite[Theorem~4.14]{ThiVil23arX:Soft} that $w\in S_\soft$. 
We get
\[
w \leq e'+z\leq y\ll x,\quad 
x'\ll y'\leq \infty e'\leq \infty w,\andSep 
u'\ll u''\leq z\wedge l\leq w.
\]

By \cite[Theorem~2.5]{AntPerRobThi21Edwards}, the map $S \to S$, $s \mapsto s\wedge l$, is additive.
Using this at the second and fourth step, we get
\begin{align*}
v'' + 2(d'\wedge l)
&= (v'' \wedge l) + 2(d'\wedge l)
= (v''+2d')\wedge l \\
&\leq (z+d')\wedge l 
= (z \wedge l) + (d'\wedge l)
\leq w + (d'\wedge l).
\end{align*}

We also have $d'\wedge l\leq (Mz)\wedge l = M(z\wedge l)\leq Mw$, and this implies that
\[
\widehat{v''}\leq \widehat{w}.
\]

Now, since $v'\ll v''$ and $\frac{k}{k+1}<1$, we can apply \cite[Lemma~2.2.5]{Rob13Cone} to obtain
\[
\frac{k}{k+1} \widehat{v'} \ll \widehat{v''}\leq \widehat{w}.
\]

Since $w$ is strongly soft and $S_\soft$ is a sub-\CuSgp{}, there exists a $\ll$-increasing sequence of soft elements with supremum $w$. Using that the rank map $x\mapsto  \widehat{x}$ preserves suprema of increasing sequences, we can find $w' \in S_\soft$ such that
\[
w'\ll w, \quad
\frac{k}{k+1} \widehat{v'} \leq \widehat{w'}, \quad 
x'\ll \infty w',\andSep 
u'\ll w'.
\]
Further, we have $w'\ll w\ll x$. 
This shows that $w'$ is a strongly soft element in $L_x$, as desired.

If, additionally, $S$ is left-soft separative, we can apply \autoref{rmk:LSSepEqu} on \eqref{eq:Separative} to obtain that $v''+d \leq z$, and so $v'' \leq z$.
We also have $v'' \leq l$ and thus
\[
v' \ll v'' \leq z \wedge l \leq w.
\]
We also have $u' \ll u'' \leq w$ and $x' \ll \infty w$.
Using that~$w$ is strongly soft and that~$S_\soft$ is a sub-\CuSgp{}, we can find $w' \in S_\soft$ such that $u',v' \ll w' \ll w$ and $x' \ll \infty w'$.
Then $w'$ has the desired properties.
\end{proof}

\begin{rmk}
\label{rmk:CaCtbBas}
The assumption of $S$ being countably based in \autoref{prp:PreLxDirected} is only used to prove the existence of the infimum $z\wedge l$. If $S$ is the Cuntz semigroup of a \ca{}, this infimum always exists; see \cite{CiuRobSan10CuIdealsQuot}. 
Thus, the first part of \autoref{prp:PreLxDirected} holds for every \ca{} with the Global Glimm Property.
\end{rmk}

\begin{prp}
\label{prp:LxDirected}
Let $S$ be a countably based, $(2,\omega )$-divisible \CuSgp{} satisfying \axiomO{5}-\axiomO{7}, let $x',x\in S$ with $x' \ll x$, let $k \in \NN$, and let $u' \in L_x$.
Then, for every finite subset $C\subseteq L_x$, there exists a strongly soft element $w'\in L_x$ such that 
\[
u'\ll w',\quad 
x'\ll \infty w',\andSep 
\frac{k}{k+1}\widehat{v'} \leq \widehat{w'} \text{ in $\LAff (F(S))$}
\]
for every $v'\in C$.
\end{prp}
\begin{proof}
We will prove the result by induction on $\vert C\vert$, the size of $C$.
If $\vert C\vert =1$, the result follows from \autoref{prp:PreLxDirected}.

Thus, fix $n\in\NN$ with $n \geq 2$, and assume that the result holds for any finite subset of $n-1$ elements.
Given $C\subseteq L_x$ with $\vert C\vert =n$, pick some $v_0\in C$. 
Applying the induction hypothesis, we get an element $w''\in L_x$ such that 
\[
u' \ll w'',\quad 
x' \ll \infty w'',\andSep 
\frac{k}{k+1}\widehat{v'} \leq \widehat{w''}
\]
for every $v'\in C \setminus \{v_0\}$.

Now, applying \autoref{prp:PreLxDirected} to $x'$, $w''$ and $v_0$, we get a strongly soft element $w' \in L_x$ such that 
\[
w'' \ll w',\quad 
x' \ll \infty w',\andSep 
\frac{k}{k+1}\widehat{v_0} \leq \widehat{w'}.
\]

Then $\widehat{w''}\leq \widehat{w'}$, which shows that $w'$ satisfies the required conditions.
\end{proof}

\begin{prp}
\label{prp:FuncSupSoft}
Let $S$ be a countably based, $(2,\omega )$-divisible \CuSgp{} satisfying \axiomO{5}-\axiomO{7}, let $x \in S$, and let $u' \in L_x$.
Then there exists $w\in S_\soft$ such that
\[
u'\ll w\leq x\leq \infty w,\andSep 
\lambda (w)=\sup_{v'\in L_x} \lambda (v'),
\]
for every $\lambda \in F(S)$.
\end{prp}
\begin{proof}
By definition of $L_x$, we obtain $u \in S_\soft$ such that $u' \ll u \ll x$. 
Let $(x_n)_n$ be a $\ll$-increasing sequence with supremum $x$, and such that $u \ll x_0$. 
Note that the sets $L_{x_n}$ form an increasing sequence of subsets of $S$ with $L_x = \bigcup_n L_{x_n}$.

Let $B$ be a countable basis for $S$.
Then
\[
B \cap L_x = \bigcup_n (B \cap L_{x_n}),
\]
and we can choose a $\subseteq$-increasing sequence $(C_n)_n$ of finite subsets of $B \cap L_x$ such that 
\[
B \cap L_x = \bigcup_n C_n, \andSep 
C_n \subseteq B \cap L_{x_n} \text{ for each $n$}.
\]

We have $u'\in L_{x_0} \subseteq L_{x_1}$.
Apply \autoref{prp:LxDirected} to $k=1,(0 \ll x_1),u',C_1$ to obtain a strongly soft element $w_1'\in L_{x_1}$ such that 
\[
u' \ll w_1',\quad 
0 \ll \infty w_1', \andSep 
\frac{1}{2}\widehat{v'} \leq \widehat{w_1'}
\]
for every $v' \in C_1$.
 
We have $w_1' \in L_{x_2}$. 
Applying \autoref{prp:LxDirected} again to $k=2,(x_1 \ll x_2),w_1',C_2$, we obtain a strongly soft element $w_2' \in L_{x_2}$ such that 
\[
w_1' \ll w_2',\quad 
x_1 \ll \infty w_2',\andSep 
\frac{2}{3} \widehat{v'}\leq \widehat{w_2'}
\]
for every $v'\in C_2$.
 
Proceeding inductively, we get a $\ll$-increasing sequence of strongly soft elements $(w_n')_n$ such that 
\[
w_n'\in L_{x_n},\quad 
x_{n-1}\ll \infty w_{n}'\andSep 
\frac{n}{n+1}\widehat{v'}\leq \widehat{w_n'}
\]
for every $v'\in C_n$ and $n \geq 2$.

Set $w:=\sup_n w_n'$, which is strongly soft by \cite[Theorem~4.14]{ThiVil23arX:Soft}. 
Note that we get $u' \ll w_1' \leq w \leq x$ by construction. 
Further, since $x_n\leq \infty w_{n+1}'\leq \infty w$ for each $n \geq 2$, we deduce that $x\leq \infty w$.
 
Now take $\lambda \in F(S)$. 
Given $v'\in B \cap L_x$, choose $n_0 \geq 2$ such that $v' \in C_{n_0}$. 
We have 
\[
\frac{n}{n+1} \lambda(v') 
\leq \lambda(w_n')
\leq \lambda(w)
\]
for every $n\geq n_0$. 
Thus, it follows that $\lambda (v') \leq \lambda (w)$ for every $v' \in B \cap L_x$.

Since $L_x$ is downward-hereditary, every element in $L_x$ is the supremum of an increasing sequence from $B \cap L_x$.
Using also that functionals preserve suprema of increasing sequences, we obtain 
\[
\sup_{v'\in L_x} \lambda (v')
\leq \sup_{v'\in B\cap L_x} \lambda (v')
\leq \lambda (w) 
= \sup_n \lambda (w_n')\leq 
\sup_{v'\in L_x} \lambda (v'),
\]
which shows that $w$ has the desired properties.
\end{proof}

\begin{lma}
\label{prp:SigmaDomination}
Let $S$ be a $(2, \omega)$-divisible \CuSgp{} satisfying \axiomO{5}-\axiomO{7}, and let $x',x,t\in S$ be such that $x'\ll x\leq \infty t$. 
Then there exists a strongly soft element $u'\in L_x$ such that 
\[
x'\ll u'+t.
\]
\end{lma}
\begin{proof}
Choose $x'' \in S$ such that $x'\ll x''\ll x$. 
Applying \cite[Proposition~4.10]{ThiVil23Glimm} to 
\[
x''\ll x\leq\infty x,\infty t,
\]
we get $s\in S$ such that
\[
x'' \ll \infty s, \andSep s\ll x,t.
\]

By \cite[Proposition~7.7]{ThiVil23arX:Soft}, we can choose $s'\in S_\soft $ such that
\[
x'' \leq \infty s', \andSep s'\ll s.
\]
Then $x'' \ll \infty s'$.
Applying \axiomO{5} to $s'\ll s\leq x$, we obtain $v\in S$ satisfying
\[
v+s'\leq x\leq v+s.
\]

In particular, one has $x''\ll v+s$. 
Applying \axiomO{6} to $x' \ll x'' \leq v+s$, we find $u\in S$ such that 
\[
 x'\ll u+s,\andSep u\ll x'' , v.
\]

Since $u\ll x''\leq \infty s'$, it follows from \cite[Theorem~4.14]{ThiVil23arX:Soft} that $u+s'$ is soft. 
Further, we get 
\[
x' 
\ll u+s
\leq u+t
\leq (u+s')+t, \andSep 
 u+s'\leq v+s'\leq x.
\]

Using that $S_\soft$ is a sub-\CuSgp{} by \autoref{prp:SoftPartAxioms}, find can find $u'\in S_\soft$ such that 
\[
x'\ll u'+t,\andSep 
u'\ll u+s'\leq x.
\]
Then $u'\in L_x$, which shows that $u'$ has the desired properties.
\end{proof}

\begin{lma}
\label{prp:ImprovSigmaDomination}
Let $S$ be a $(2,\omega )$-divisible \CuSgp{} satisfying \axiomO{5}-\axiomO{7}, and let $t\in S_\soft$ and $t',x',x\in S$ be such that 
\[
x'\ll x\leq \infty t,\andSep 
t'\ll t.
\]

Then, there exists a strongly soft element $v'\in L_x$ such that 
\[
x'+t'\leq v'+t.
\]
\end{lma}
\begin{proof}
By \cite[Proposition~4.13]{ThiVil23arX:Soft}, there exists $s\in S_\soft$ such that 
\[
t'+s\leq t\leq \infty s.
\]
Applying \autoref{prp:SigmaDomination} to $x'\ll x\leq \infty s$, we obtain a strongly soft element $v'\in L_x$ satisfying $x'\leq v'+s$.
Consequently, we obtain 
\[
x'+t'\leq v'+s+t'\leq v'+t. \qedhere
\]
\end{proof}

We refer to \cite[Section~5]{ThiVil21DimCu2} for an introduction to the basic technique to reduce certain proofs about \CuSgp{s} to the countably based setting. 
In particular, a property $\mathcal{P}$ for \CuSgp{s} is said to satisfy the L\"{o}wenheim-Skolem condition if, for every \CuSgp{} $S$ satisfying $\mathcal{P}$, there exists a $\sigma$-complete and cofinal subcollection of countably based sub-\CuSgp{s} of $S$ satisfying $\mathcal{P}$.

\begin{lma}
\label{prp:PassingSoftElementToSubCuSgp}
Let $S$ be a \CuSgp{}, let $u \in S_\soft$, and let $\mathcal{R}$ be the family of countably based sub-\CuSgp{s} $T \subseteq S$ containing $u$ and such that $u$ is strongly soft in $T$.
Then $\mathcal{R}$ is $\sigma$-complete and cofinal.
\end{lma}
\begin{proof}
Strong softness is preserved under \CuMor{s}, and the inclusion map of a sub-\CuSgp{} is a \CuMor.
Hence, given sub-\CuSgp{s} $T_1 \subseteq T_2 \subseteq S$ containing $u$, if $u$ is strongly soft in $T_1$ then it is also strongly soft in $T_2$.
This implies in particular that $\mathcal{R}$ is $\sigma$-complete.

To show that $\mathcal{R}$ is cofinal, let $T_0 \subseteq S$ be a countably based sub-\CuSgp, and let $B_0 \subseteq T_0$ be a countable basis, that is, a countable subset such that every element in $T_0$ is the supremum of an increasing sequence from $B_0$.

Let $(u_n)_n$ be a $\ll$-increasing sequence in $S$ with supremum $u$.
Since $u$ is strongly soft in $S$, for each $n$ we obtain $t_n \in S$ such that 
\[
u_n + t_n \ll u, \andSep
u_n \ll \infty t_n.
\]

By \cite[Lemma~5.1]{ThiVil21DimCu2}, there exists a countably based sub-\CuSgp{} $T \subseteq S$ containing 
\[
B_0 \cup \{u_0,u_1,\ldots\} \cup \{t_0,t_1,\ldots\}.
\]
One checks that $T_0 \subseteq T$, and that $u$ is strongly soft in $T$.
\end{proof}

\begin{thm}
\label{prp:SoftRankS}
Let $S$ be a $(2,\omega )$-divisible \CuSgp{} satisfying \axiomO{5}-\axiomO{7}, let $x\in S$, and let $u'\in L_x$.
Then there exists $w\in S_\soft$ such that
\[
u'\ll w\leq x\leq \infty w,\andSep 
\widehat{w} = \widehat{x}.
\]
\end{thm}
\begin{proof}
We first prove the result under the additional assumption that $S$ is countably based. 
Use \autoref{prp:FuncSupSoft} to obtain $w\in S_\soft$ such that
\[
u' \ll w \leq x \leq \infty w, \andSep
\lambda (w)=\sup_{v'\in L_x} \lambda (v'),
\]
for every $\lambda \in F(S)$.
Since $w\leq x$, we have $\widehat{w} \leq \widehat{x}$.
To show the reverse inequality, let $\lambda\in F(S)$.
We need to prove that $\lambda(x) \leq \lambda(w)$.
 
Take $x',w'\in S$ such that $x' \ll x$ and $w'\ll w$. 
Applying \autoref{prp:ImprovSigmaDomination}, we obtain an element $v'\in L_x$ such that 
\[
x'+w'\leq v'+w.
\]
 
Since $v'$ belongs to $L_x$, we have $\lambda (v')\leq \lambda (w)$. 
This implies
\[
\lambda (x')+\lambda (w')\leq \lambda (v')+\lambda (w)
\leq 2\lambda (w).
\]
Passing to the supremum over all $x'$ way-below $x$, and all $w'$ way-below $w$, we get
\[
\lambda (x)+\lambda (w)
\leq 2\lambda (w).
\]

This proves $\lambda (x)\leq \lambda (w)$.
Indeed, if $\lambda (w)=\infty$, then there is nothing to prove. 
If $\lambda (w)\neq \infty$, we can cancel $\lambda (w)$ from the previous inequality.

\medskip

We now consider the case that $S$ is not countably based.
Choose $u \in S_\soft$ such that $u' \ll u \ll x$.
Since $(2,\omega)$-divisibility and \axiomO{5}-\axiomO{7} each satisfy the L{\"o}wenheim-Skolem condition, and using also \autoref{prp:PassingSoftElementToSubCuSgp}, we can use the technique from \cite[Section~5]{ThiVil21DimCu2} to deduce that there exists a countably based, $(2,\omega )$-divisible sub-\CuSgp{} $H \subseteq S$ satisfying \axiomO{5}-\axiomO{7}, containing $x$, $u$ and $u'$, and such that $u$ is strongly soft in $H$.

Applying the first part of the proof to $H$, we find $w\in H_\soft$ such that
\[
u'\ll w\leq x\leq \infty w,\andSep 
\lambda (x)=\lambda (w)
\]
for every $\lambda \in F(H)$.
 
Since the inclusion $\iota\colon H\to S$ is a \CuMor{}, it follows that $w$ is strongly soft in $S$. 
Further, any functional $\lambda$ on $S$ induces the functional $\lambda\iota$ on $H$. 
This shows that $w$ satisfies the required conditions.
\end{proof}

\begin{thm}
\label{prp:SoftRankCuGGP}
Let $A$ be a stable \ca{} with the Global Glimm Property. 
Then, for any $a\in A_+$ there exists a soft element $b \in A_+$ with $b\precsim a$ and such that 
\[
d_\tau(a) = d_\tau(b)
\]
for every $\tau \in \QT(A)$.
\end{thm}
\begin{proof}
Let $a\in A_+$. Since $A$ has the Global Glimm Property, it follows from \cite[Theorem~3.6]{ThiVil23Glimm} that $\Cu (A)$ is $(2,\omega )$-divisible. 
Using \autoref{prp:SoftRankS}, find $w\in \Cu (A)_\soft$ such that $w\leq [a]$ and $\lambda (w)=\lambda ([a])$ for every $\lambda\in F(\Cu (A))$.

By \autoref{prp:SoftStSoftEq}, there exists a soft element $b\in A_+$ such that $w=[b]$. The result now follows from the fact that the map 
\[
\tau\mapsto \left([a]\mapsto d_\tau(a)\right)
\]
is a natural bijection from ${\rm QT} (A)$ to $F(\Cu (A))$;
see \cite[Theorem~4.4]{EllRobSan11Cone}.
\end{proof}

\begin{lma}
\label{prp:EqSupSoft}
Let $S$ be a $(2,\omega )$-divisible \CuSgp{} $S$ satisfying \axiomO{5}, let $x\in S$, and let $\lambda \in F(S)$.
Then
\[
\sup_{\{ v\in S_\soft : v\leq x \}}\lambda (v)=\sup_{v'\in L_x} \lambda (v')
\]
\end{lma}
\begin{proof}
Given $v' \in L_x$, there exists $v \in S_\soft$ with $v' \leq v \leq x$, which shows the inequality `$\geq$'.

Conversely, let $v \in S_\soft$ with $v \leq x$.
Since $S_\soft$ is a sub-\CuSgp{} by \autoref{prp:SoftPartAxioms}, there exists a $\ll$-increasing sequence $(v_n')_n$ in $S_\soft$ with supremum~$v$.
Each $v_n'$ belongs to $L_x$, and one gets 
\[
\lambda(v) 
= \sup_n \lambda(v_n') 
\leq \sup_{v'\in L_x} \lambda (v').
\]
This shows the the inequality `$\leq$'.
\end{proof}

We will prove in \autoref{prp:FScongFSSoft} that the inclusion $\iota\colon S_\soft\to S$ induces a homeomorphism $\iota^*\colon F(S)\to F(S_\soft)$. 
The inverse of $\iota^*$ is constructed in the next result.

\begin{prp}
\label{prp:SoftifiedLambda}
Let $S$ be a $(2,\omega)$-divisible \CuSgp{} satisfying \axiomO{5}-\axiomO{7}, and let $\lambda \in F(S_\soft )$. 
Then $\lambda_\soft \colon S \to [0,\infty]$ given by
\[
\lambda_\soft(x) := \sup_{\{ v\in S_\soft : v\leq x \}} \lambda (v)
\]
for $x \in S$, is a functional on $S$.
\end{prp}
\begin{proof}
It is easy to see that $\lambda_\soft$ preserves order.
Further, given an increasing sequence $(x_n)_n$ with supremum $x$ in $S$, we have that for every $v'\in L_x$ there exists $n\in\NN$ with $v'\in L_{x_n}$. 
Thus, using \autoref{prp:EqSupSoft}, we get
\[
\lambda_\soft (x) 
= \sup_{v'\in L_x}\lambda (v')
\leq \sup_n \left(\sup_{v'\in L_{x_n}}\lambda (v') \right)
=\sup_n \lambda_\soft (x_n).
\]

Since $\lambda_\soft$ is order-preserving, we also have $\sup_n \lambda_\soft (x_n)\leq \lambda_\soft (x)$, which shows that $\lambda_\soft$ preserves suprema of increasing sequences.

\medskip
 
Given $x,y\in S$ and $u,v\in S_\soft$ such that $u\leq x$ and $v\leq y$, we have $u+v\in S_\soft$ and $u+v\leq x+y$. 
This implies that 
\[
\lambda_\soft (x)+\lambda_\soft (y)\leq \lambda_\soft (x+y).
\]
Thus, $\lambda_\soft$ is subadditive.

\medskip

Finally, we show that $\lambda_\soft$ is superadditive.
Given $x,y\in S$ and $w'\in L_{x+y}$, take $x',x'',y',y''\in S$ such that 
\[
x'\ll x''\ll x,\quad
y'\ll y''\ll y,\andSep
w'\ll x'+y'.
\]

By \cite[Proposition~7.7]{ThiVil23arX:Soft}, there exist $s,t\in S_\soft$ such that
\[
s\leq x''\leq \infty s, \andSep
t\leq y''\leq \infty t.
\]
 
Take $s',t'\in S$ such that $s'\ll s$ and $t'\ll t$. 
Using \autoref{prp:ImprovSigmaDomination}, we find $u'\in L_x$ and $v'\in L_y$ such that 
\[
x'+s'\leq u'+s,\andSep 
y'+t'\leq v'+t.
\]
 
Consequently, one has 
\[
w'+s'+t'
\leq x'+y'+s'+t'
\leq u'+s+v'+t.
\]
 
Applying \autoref{prp:SoftRankS}, find $u,v\in S_\soft$ such that 
\[
u'\ll u\leq x\leq \infty u,\andSep 
v'\ll v\leq y\leq \infty v.
\]
 
This implies
\[
w'+s'+t'
\leq u+s+v+t
\] 
and, therefore, 
\[
\lambda (w')+\lambda (s' + t')
\leq \lambda (u)+\lambda (v)+\lambda (s+t).
\]

Passing to the suprema over all $s'$ way-below $s$, and all $t'$ way-below $t$, we deduce that 
\[
\lambda (w')+\lambda (s + t)
\leq \lambda (u)+\lambda (v)+\lambda (s+t).
\]
Note that $s+t\leq x''+y'' \ll x+y \leq \infty (u+v)$.
This allows us to cancel $\lambda (s + t)$, and we obtain
\[
\lambda (w')
\leq \lambda (u)+\lambda (v)
\leq \lambda_\soft (x)+\lambda_\soft (y).
\]

Since this holds for every $w'\in L_{x+y}$, we can apply \autoref{prp:EqSupSoft} to get
\[
\lambda_\soft(x+y)
= \sup_{\{ w\in S_\soft : w \leq x+y \}}\lambda(w)
= \sup_{w'\in L_{x+y}} \lambda(w')
\leq \lambda_\soft (x)+\lambda_\soft (y).
\]
This show that $\lambda_\soft$ is superadditive, and thus a functional.
\end{proof}

\begin{thm}
\label{prp:FScongFSSoft}
Let $S$ be a $(2,\omega)$-divisible \CuSgp{} satisfying \axiomO{5}-\axiomO{7}. 
Let $\iota\colon S_\soft\to S$ be the canonical inclusion. 
Then the map $\iota^*\colon F(S)\to F(S_\soft)$ given by $\iota^*(\lambda):=\lambda\circ\iota$ is a natural homemomorphism.
\end{thm}
\begin{proof}
Given $\lambda \in F(S_\soft)$, let $\lambda_\soft \in F(S)$ be defined as in \autoref{prp:SoftifiedLambda}.
This defines a map $\phi \colon F(S_\soft )\to F(S)$ by $\phi(\lambda) :=\lambda_\soft$. 
We verify that $\iota^*\phi=\id_{F(S_\soft )}$ and $\phi\iota^*=\id_{F(S)}$. 

Given $\lambda \in F(S_\soft )$ and $w\in S_\soft$, we have 
\[
\iota^*\phi(\lambda )(w) 
= \iota^*\lambda_\soft (w) 
= \lambda_{\soft} (\iota (w))
= \sup_{\{ v\in S_\soft : v\leq w \}} \lambda (v) 
= \lambda (w),
\]
which shows $\iota^*\phi=\id_{F(S_\soft )}$.
 
Conversely, if $\lambda \in F(S)$ and $x\in S$, we can use \autoref{prp:SoftRankS} at the last step to obtain 
\[
\phi\iota^*(\lambda)(x)
= \phi (\lambda \iota )(x) 
= \sup_{\{ v\in S_\soft : v\leq x \}}\lambda (v) 
= \lambda (x).
\]

This shows that $\iota^*$ is a bijective, continuous map. 
Since $F(S)$ and $F(S)_\soft$  are both compact, Hausdorff spaces, it follows that $\iota^*$ is a homemorphism.
\end{proof}

Since simple, nonelementary \ca{s} automatically have the Global Glimm Property, the next result can be considered as a generalization of \cite[Lemma~3.8]{Phi14arX:LargeSub} to the non-simple setting.

\begin{thm}
\label{prp:QTAEqFCuASof}
Let $A$ be a \ca{} with the Global Glimm Property. 
Then $\QT(A)$ is naturally homemomorphic to $F(\Cu (A)_\soft )$.
\end{thm}
\begin{proof}
The result follows from \autoref{prp:FScongFSSoft} and the fact that $\QT(A)$ is naturally homemomorphic to $F(\Cu (A))$; 
see \cite[Theorem~4.4]{EllRobSan11Cone}.
\end{proof}

\section{Retraction onto the soft part of a Cuntz semigroup}
\label{sec:retract}

Let $S$ be a countably based, left-soft separative, $(2,\omega)$-divisible \CuSgp{} satisfying \axiomO{5}-\axiomO{7}. 
Given any $x\in S$, we have seen in \autoref{prp:PreLxDirected} that $L_x$ is upward-directed. 
It then follows from \cite[Remarks~3.1.3]{AntPerThi18TensorProdCu} that $L_x$ has a supremum, which justifies the following:

\begin{dfn}
\label{dfn:sigma}
Let $S$ be a countably based, left-soft separative, $(2,\omega )$-divisible \CuSgp{} satisfying \axiomO{5}-\axiomO{7}. 
We define $\sigma \colon S \to S$ by 
\[
\sigma (x) 
:= \sup L_x 
= \sup \big\{ u'\in S : u' \ll u \ll x \text{ for some } u\in S_\soft \big\}
\]
for $x \in S$.
\end{dfn}

We will see in \autoref{prp:sigmaSoft} that $\sigma(x)$ is the largest strongly soft element dominated by $x$.
Therefore, we often view $\sigma$ as a map $S \to S_\soft$.
In \autoref{prp:SsoftGenCuMor} we show that $\sigma$ is close to being a generalized \CuMor{}, and in \autoref{prp:SigmaCuMor} we give sufficient conditions ensuring that it is.

If $A$ is a separable \ca{} satisfying the Global Glimm Property and with left-soft separative Cuntz semigroup, then $\Cu(A)$ satisfies the assumptions of \autoref{dfn:sigma}. If $A$ also has stable rank one or strict comparison of positive elements, then $\sigma \colon \Cu(A) \to \Cu(A)_\soft$ is a generalized \CuMor{};
see \autoref{prp:SigmaCuMorCStar}. 
Then $\Cu(A)_\soft$ is a \emph{retract} of $S$;
see \autoref{dfn:retract}.
This generalizes the construction of predecessors in the context of simple \ca{s} from \cite{Eng14PhD}, as well as the constructions from \cite[Section~5.4]{AntPerThi18TensorProdCu} and \cite[Proposition~2.9]{Thi20RksOps}.

\begin{rmk}
Let $S$ be a weakly cancellative \CuSgp{} satisfying \axiomO{5}-\axiomO{7} (for instance, the Cuntz semigroup of a stable rank one \ca{}). 
Take $x\in S$, and consider the set
\[
L_x':= \big\{ u' : u'\ll u \leq \infty s, \text{and } u+s\ll x \text{ for some } u,s\in S \big\}.
\]
 
A slight modification of \autoref{prp:LxDirected} shows that $L_x'$ is upward directed.

If $S$ is countably based and $(2,\omega )$-divisible, it is readily checked that $\sigma (x) = \sup L_x = \sup L_x'$. 
However, if $S$ is not $(2,\omega )$-divisible, $\sup L_x'$ may not be strongly soft. For example, the Cuntz semigroup of $\mathbb{C}$ is $\NNbar = \NN\cup\{\infty \}$, which is weakly cancellative. 
One can check that 
\[
\sup L_x' = \begin{cases}
0, & \text{ if } x=0 \\
x-1, & \text{ if } x\neq 0,\infty \\
\infty, & \text{ if } x=\infty \\
\end{cases}.
\]
In particular, if $x\neq 0, \infty$, we get $\sup L_x'=x-1$, which is not strongly soft.

As another example, there are \CuSgp{s} whose order structure is deeply related to its soft elements but where $\sup L_x'$ is rarely strongly soft: 
Let $S$ be a \CuSgp{} of the form $\Lsc(X,\NNbar)$ for some $T_1$-space $X$ (these were called \emph{Lsc-like} in \cite{Vil21arX:CommCuAI}).
An element $f \in \Lsc(X,\NNbar)$ is strongly soft if and only if $f = \infty \chi_U$ for the indicator function $\chi_U$ of some open subset $U \subseteq X$.
Thus, if $x\in S$ satisfies $x\ll \infty$, we have $\sup L_x'\ll \infty$, which implies that $\sup L_x'$ is not strongly soft, unless it is zero.
\end{rmk}

\begin{prp}
\label{prp:sigmaSoft}
Let $S$ be a countably based, left-soft separative, $(2, \omega)$-divisible \CuSgp{} satisfying \axiomO{5}-\axiomO{7}, and let $x\in S$. 
Then:
\begin{enumerate}
\item
The element $\sigma (x)$ is the largest strongly soft element dominated by $x$.
\item
We have $\infty x=\infty\sigma (x)$.
\item
We have $x=\sigma (x)$ if and only if $x$ is strongly soft.
\item
We have $x\leq \sigma (x)+t$ for all $t \in S$ with $x\leq \infty t$.
\end{enumerate}
\end{prp}
\begin{proof}
To verify~(1), note that the members of $L_x$ are bounded by $x$, and consequently $\sigma(x) \leq x$.
To see that $\sigma (x)$ is strongly soft, let $s\in S$ be such that $s\ll \sigma (x)$.
We will find $t \in S$ such that $s + t \ll \sigma(x)$ and $s \ll \infty t$.

Since $\sigma(x)=\sup L_x$, there exists $u' \in L_x$ such that $s \ll u' \leq \sigma (x)$. 
Using that $u'\in L_x$, we find $u\in S_\soft$ with $u'\ll u\ll x$.
By \autoref{prp:SoftPartAxioms}, $S_\soft$ is a sub-\CuSgp, and we obtain $u'' \in S_\soft$ such that
\[
s \ll u' \ll u'' \ll  u \ll x.
\]
Then $s \ll u'' \in S_\soft$ and by definition we obtain $t \in S$ such that $s+t\ll u''$ and $s \ll \infty t$.
We have $u'' \in L_x$ and therefore $u'' \leq \sigma(x)$, which shows that $t$ has the desired properties.

Thus, $\sigma(x)$ is a strongly soft element dominated by $x$.
To show that it is the largest element with these properties, let $w \in S_\soft$ satisfy $w\leq x$.
We can use once again that $S_\soft$ is a sub-\CuSgp{} to find a $\ll$-increasing sequence $(w_n)_n$ of strongly soft elements with supremum~$w$.
Then $w_n \in L_x$ for each $n$, and consequently
\[
w 
= \sup_n w_n 
\leq \sup L_x
= \sigma (x). 
\]

This also shows that $x=\sigma (x)$ if and only if $x$ is strongly soft. 
We have proved~(1) and~(3).

To verify~(2), we first note that $\infty \sigma(x) \leq \infty x$ since $\sigma(x) \leq x$.
For the converse inequality, use \autoref{prp:FuncSupSoft} to obtain $w\in S_\soft$ with $w\leq x\leq \infty w$. 
By~(1), we have $w\leq \sigma (x)$, and we get
\[
\infty x= \infty w\leq \infty \sigma(x).
\]

Finally, to prove~(4), let $t \in S$ satisfy $x \leq \infty t$.
Let $x' \in S$ satisfy $x' \ll x$.
Applying \autoref{prp:SigmaDomination}, we obtain $u' \in L_x$ such that $x' \ll u'+t$.
Then
\[
x' 
\ll u'+t
\leq \sigma(x)+t.
\]
Passing to the supremum over all $x'$ way-below $x$, we get $x \leq \sigma(x)+t$, as desired.
\end{proof}

\begin{exa}
Let $A$ be a separable, $\mathcal{W}$-stable \ca{}, that is, $A \cong A \otimes \mathcal{W}$ where $\mathcal{W}$ denotes the Jacelon-Razak algebra.
Then, every element in $\Cu (A)$ is strongly soft.
Thus \autoref{prp:sigmaSoft} implies that $\sigma (x)=x$ for every $x\in \Cu (A)$.
We refer to \cite[Section~7.5]{AntPerThi18TensorProdCu} for details.
 
Similarly, given a separable $\mathcal{Z}$-stable \ca{} $A$, where $\mathcal{Z}$ denotes the Jiang-Su algebra, then it follows from \cite[Theorem~7.3.11]{AntPerThi18TensorProdCu} that $\Cu (A)$ has $Z$-multiplication.
Here, $Z=(0,\infty]\sqcup \NN$ is the Cuntz semigroup of $\mathcal{Z}$, and $(0,\infty]$ is the subsemigroup of nonzero, strongly soft elements.
Let $1' \in Z$ be the strongly soft element corresponding to $1\in [0,\infty ]$. 
As noted in \cite[Proposition~7.3.16]{AntPerThi18TensorProdCu}, one has 
\[
1'\Cu (A) 
= \Cu (A)_\soft
\cong \Cu (A)\otimes [0,\infty ].
\] 
This implies that $\sigma (x) = 1'x$ for each $x\in\Cu (A)$.
\end{exa}

\begin{lma}
\label{prp:2Soft}
Let $S$ be a countably based, left-soft separative, $(2,\omega )$-divisible \CuSgp{} satisfying \axiomO{5}-\axiomO{7}, and let $x\in S$. 
Then
\[
2\sigma (x) = x+\sigma (x).
\]
\end{lma}
\begin{proof}
Using that $\sigma (x)\leq x$, we have $2\sigma (x)\leq x+\sigma (x)$.
To show the reverse inequality, let $w \in S$ satisfy $w\ll \sigma (x)$. 
Since $\sigma (x)$ is strongly soft, it follows from \cite[Proposition~4.13]{ThiVil23arX:Soft} that there exists $t\in S$ with $w+t\leq \sigma (x)\leq \infty t$. 

We have $x \leq \infty \sigma (x)$ by \autoref{prp:sigmaSoft}~(2), and thus $x \leq \infty t$.
Therefore, $x\leq \sigma (x)+t$ by \autoref{prp:sigmaSoft}~(4).
Thus, we have 
\[
x+w\leq \sigma (x) + t +w \leq 2\sigma (x).
\]
Passing to the supremum over all $w$ way-below $\sigma (x)$, we get $x+\sigma (x)\leq 2\sigma (x)$.
\end{proof}

\begin{thm}
\label{prp:SsoftGenCuMor}
Let $S$ be a countably based, left-soft separative, $(2,\omega)$-divisible \CuSgp{} satisfying \axiomO{5}-\axiomO{7}. 
Then, the map $\sigma\colon S\to S_\soft$ preserves order, suprema of increasing sequences, and is superadditive. Further, we have 
\[
2\sigma (x+y) 
= \sigma (x+y) + \big( \sigma (x) + \sigma (y) \big) 
= 2\big( \sigma (x) + \sigma (y) \big)
\]
for every $x,y\in S$.
\end{thm}
\begin{proof}
To show that $\sigma$ is order-preserving, let $x,y\in S$ satisfy $x\leq y$. 
Then $L_x \subseteq L_y$, and thus
\[
\sigma(x) 
= \sup L_x 
\leq \sup L_y 
= \sigma(y).
\]

To show that $\sigma$ preserves suprema of increasing sequences, let $(x_n)_n$ be an increasing sequence in $S$ with supremum $x$.
Since $\sigma$ is order-preserving, one gets $\sup_n\sigma(x_n) \leq \sigma(x)$. 
Conversely, given $u' \in L_x$, choose $u \in S_\soft$ with $u' \ll u \ll x$.
Then there exists $n\in \NN$ such that $u \ll x_n$, and thus $u' \in L_{x_n}$.
We deduce that
\[
u' 
\leq \sup L_{x_n} 
= \sigma(x_n)
\leq \sup_n \sigma(x_n).
\]
Hence, $\sigma(x) = \sup L_x \leq \sup_n \sigma(x_n)$, as desired.

To see that $\sigma$ is superadditive, let $x,y\in S$. 
Note that $\sigma (x)+\sigma (y)$ is a strongly soft element bounded by $x+y$. 
Using \autoref{prp:sigmaSoft}~(1), we get $\sigma (x)+\sigma (y)\leq \sigma (x+y)$.

\medskip

Next, given $x,y \in S$, let us show that $2\sigma (x+y) \leq 2\sigma (x) + 2\sigma (y)$.
To prove this, let $w \in S$ satisfy $w \ll \sigma(x+y)$.
By \cite[Proposition~4.13]{ThiVil23arX:Soft}, there exists $s\in S$ satisfying 
\[
w + s \leq \sigma(x+y) \leq \infty s.
\]
Applying \cite[Proposition~7.7]{ThiVil23arX:Soft}, we find $t \in S$ such that $2t \leq s\leq \infty t$. 
Using also \autoref{prp:sigmaSoft}~(2), we deduce that
\[
w+2t
\leq w+s 
\leq \sigma(x+y), \andSep
x,y 
\leq \infty (x+y)
= \infty \sigma (x+y)
\leq \infty s 
\leq \infty t.
\]
Using \autoref{prp:sigmaSoft}~(4) at the second step, and \autoref{prp:2Soft} at last step, we get
\[
\begin{split}
\sigma (x+y) + w &\leq x+y+w\leq \sigma (x)+\sigma (y)+w+2t\leq 
\sigma (x)+\sigma (y) + \sigma (x+y)\\
&\leq \sigma (x)+\sigma (y) + x + y
= 2\sigma (x) + 2\sigma (y).
\end{split}
\]

Passing to the supremum over all elements $w$ way-below $\sigma(x+y)$, we obtain 
\[
2\sigma (x+y) \leq 2\sigma (x) + 2\sigma (y).
\]

Next, given $x,y \in S$, using the above inequality together with the established superadditivity of $\sigma$, we get 
\[
2\sigma (x+y) 
\leq 2\sigma (x) + 2\sigma (y)
\leq \sigma (x+y) + \big( \sigma (x) + \sigma (y) \big)
\leq 2\sigma (x+y),
\]
as desired.
\end{proof}

Recall that a \emph{generalized \CuMor} is a monoid morphism between \CuSgp{s} that preserves order and suprema of increasing sequences.
We recall the definition of \emph{retract} from  \cite[Definition~3.14]{ThiVil22DimCu}.

\begin{dfn}
\label{dfn:retract}
Let $S, T$ be \CuSgp{s}. 
We say that $S$ is a \emph{retract} of $T$ if there exist a \CuMor{} $\iota\colon S\to T$ and a generalized \CuMor{} $\sigma \colon T\to S$ such that $\sigma\circ \iota = \id_S$.
\end{dfn}

\begin{prp}
\label{prp:SigmaCuMor}
Let $S$ be a countably based, left-soft separative, $(2,\omega)$-divisible \CuSgp{} satisfying \axiomO{5}-\axiomO{7}. 
Additionally, assume one of the following:
\begin{itemize}
\item[(i)] 
$S$ is almost unperforated;
\item[(ii)] 
$S$ is inf-semilattice ordered;
\item[(iii)] 
$S\otimes \{ 0,\infty \}$ is algebraic.
\end{itemize}
Then, $\sigma $ is a generalized \CuMor{} and $S_\soft$ is a retract of $S$.
\end{prp}
\begin{proof}
By \autoref{prp:SsoftGenCuMor}, we only need to check that $\sigma$ is subadditive.
 
(i):
If $S$ is almost unperforated, then it follows from \autoref{prp:PerfCond} that $S_\soft$ is unperforated. Given any pair $x,y\in S$, we know from \autoref{prp:SsoftGenCuMor} that 
\[
2\sigma (x+y) = 2\big(\sigma (x) + \sigma (y)\big).
\]
Since this equality is in $S_\soft$, it follows that $\sigma (x+y)=\sigma (x) + \sigma (y)$.
 
For (ii) and (iii), note that it is enough to prove that $\sigma (x+y)\leq x+\sigma (y)$ for all $x,y \in S$.
Indeed, if this inequality holds, one can use it at the second and last steps to get 
\[
\sigma (x+y) = \sigma (\sigma (x+y))\leq \sigma (x+\sigma (y)) = 
\sigma (\sigma (y) +x)\leq \sigma (y)+\sigma (x),
\]
as required.

Given $x,y \in S$, we proceed to verify that $\sigma (x+y)\leq x+\sigma (y)$.
Let $w\in S$ satisfy $w \ll \sigma(x+y)$.
Choose $y'\in S$ such that
\[
y'\ll y, \andSep 
w \ll x+y'. 
\]

Since $\sigma(x+y)$ is strongly soft, it follows from \cite[Proposition~4.13]{ThiVil23arX:Soft} that there exists $r\in S_\soft$ such that 
\[
w+r \leq \sigma (x+y) \leq \infty r.
\]
Applying \autoref{prp:sigmaSoft}~(2), one gets
\[
y' \ll y \leq \infty \sigma (x+y)\leq \infty r.
\]

Applying \cite[Proposition~4.7]{ThiVil23Glimm}, we obtain $t',t\in S$ such that
\[
y'\leq\infty t', \andSep
t'\ll t\ll r,y.
\]

Using that $S$ is $(2,\omega )$-divisible, it follows from \cite[Proposition~5.6]{ThiVil23arX:Soft} that we may assume both $t'$ and $t$ to be strongly soft. 
Thus, as in the proof of \autoref{prp:O5Appl}, we can apply \axiomO{5} to obtain an element $b$ satisfying
\[
t'+b\leq y\leq t+b,\andSep 
y\leq \infty b,
\]
which implies
\[
w+r\leq\sigma(x+y)\leq x+y \leq x+t+b
\]
with $t\ll r \leq \infty (x+y)=\infty (x+b)$.

Thus, since both $w$ and $r$ are strongly soft, left-soft separativity (in the form of \autoref{prp:CancelSoft}) implies that $w\leq x+b$.
Since $S$ is countably based and satisfies \axiomO{7}, the infimum $(b\wedge\infty t')$ exists.
Note that $(b\wedge\infty t') + t'$ is soft because $(b\wedge\infty t') \leq \infty t'$; see \cite[Theorem~4.14]{ThiVil23arX:Soft}. 
Then
\[
(b\wedge\infty t') + t'
\leq b+t'
\leq y,
\]
and thus $b\wedge\infty t' \leq (b\wedge\infty t') + t' \leq \sigma(y)$ by \autoref{prp:sigmaSoft}~(1).

\medskip

(ii): Assuming that $S$ is inf-semilattice ordered, it now follows that
\[
w\leq (x+b)\wedge(x+\infty t')
= x+(b\wedge\infty t')
\leq x+\sigma(y).
\]
Passing to the supremum over all $w$ way-below $\sigma(x+y)$, we get $\sigma(x+y)\leq x+\sigma(y)$, as desired.
This proves the case~(ii).
 
\medskip

(iii):
Let us additionally assume that $y\ll \infty y$. 
Then, given $w$ and $r$ as before, we have that $y\ll \infty y\leq \infty r$. This implies that there exists $r'\in S$ such that $r'\ll r$ and $y\leq \infty r'$.
Using \autoref{prp:sigmaSoft} at the last step, one gets 
\[
w+r\leq \sigma (x+y)\leq x+y\leq x+\sigma (y)+r'
\]
with $r'\ll r \leq \infty (x+y)=\infty (x+\sigma (y))$.

Therefore, we can use \autoref{rmk:LSSepEqu} to deduce that $w\leq x+\sigma (y)$. 
Since this holds for every $w$ way-below $\sigma (x+y)$, it follows that $\sigma(x+y)\leq x+\sigma(y)$ whenever $y\ll \infty y$.

If $S\otimes \{ 0,\infty \}$ is algebraic, then by \cite[Lemma~4.16]{ThiVil23Glimm} every $y \in S$ is the supremum of an increasing sequence $(y_n)_n$ of elements $y_n \in S$ such that $y_n \ll \infty y_n$.
Using the above for each $y_n$, and using that $\sigma$ preserves suprema of increasing sequences, we get
\[
\sigma(x+y)
= \sup_n \sigma(x+y_n)
\leq \sup_n \big( x+\sigma(y_n) \big)
= x + \sigma(y),
\]
as desired.
\end{proof}

\begin{thm}
\label{prp:SigmaCuMorCStar}
Let $A$ be a separable \ca{} with the Global Glimm Property. 
Additionally, assume one of the following:
\begin{itemize}
\item[(i)] 
$A$ has strict comparison of positive elements;
\item[(ii)] 
$A$ has stable rank one;
\item[(iii)] 
$A$ has topological dimension zero, and $\Cu (A)$ is left-soft separative.
\end{itemize}
Then, $\Cu (A)_\soft$ is a retract of $\Cu (A)$.
\end{thm}
\begin{proof}
The Cuntz semigroup $\Cu(A)$ is countably based and satisfies \axiomO{5}-\axiomO{7}.
Since $A$ has the Global Glimm Property, it follows from \cite[Theorem~3.6]{ThiVil23Glimm} that $\Cu (A)$ is $(2,\omega )$-divisible.
We check that the additional conditions of \autoref{prp:SigmaCuMor} are satisfied:

(i):
Assume that $A$ has has strict comparison of positive elements.
Then $\Cu(A)$ is almost unperforated by \cite[Proposition~6.2]{EllRobSan11Cone}, and left-soft separative by \autoref{prp:StCompLSSep}.
This verifies \autoref{prp:SigmaCuMor}~(i).

(ii) :
Assume that $A$ has stable rank one.
Then $\Cu(A)$ is inf-semilattice ordered by \cite[Theorem~3.8]{AntPerRobThi22CuntzSR1}, and left-soft separative by \autoref{prp:WCimpLSSep}.
This verifies \autoref{prp:SigmaCuMor}~(ii).
 
(iii):
Assume that $A$ has topological dimension zero, and $\Cu (A)$ is left-soft separative.
Then $\Cu (A)\otimes\{ 0,\infty \}$ is algebraic by \cite[Proposition~4.18]{ThiVil23Glimm}.
This verifies \autoref{prp:SigmaCuMor}~(iii).
\end{proof}

\begin{qst}
Let $S$ be a countably based, weakly cancellative, $(2,\omega)$-divisible \CuSgp{} satisfying \axiomO{5}-\axiomO{7}.
Is the map $\sigma\colon S\to S_\soft$ subadditive?
\end{qst}

With view towards the proof of subadditivity in \autoref{prp:SsoftGenCuMor}, we ask:

\begin{qst}
\label{qst:InfimumIdempotent}
Let $S$ be the Cuntz semigroup of a \ca{}.
Let $x,y,z,w\in S$ satisfy
\[
w=2w, \quad x\leq y+z, \andSep x\leq y+w.
\]

We know that $z\wedge w$ exists.
Does it follow that $x\leq y+(z\wedge w)$?
\end{qst}

\autoref{qst:InfimumIdempotent} above has a positive answer if $S$ satisfies the \emph{interval axiom}, as defined in \cite[Definition~9.3]{ThiVil21arX:NowhereScattered}.

\section{Dimension of a Cuntz semigroup and its soft part}
\label{sec:DimSoft}

Let $S$ be a countably based, left-soft separative, $(2,\omega)$-divisible \CuSgp{} satisfying \axiomO{5}-\axiomO{7}, and assume that $\sigma\colon S\to S_\soft$ is a generalized \CuMor{}. 
We show that the (covering) dimension of $S$ and $S_\soft$, as defined in \cite[Definition~3.1]{ThiVil22DimCu}, are closely related:
We have $\dim(S_\soft)\leq\dim(S)\leq\dim(S_\soft)+1$; 
see \autoref{prp:SsoftDim}.

Using the technique developed in \cite[Section~5]{ThiVil21DimCu2}, we remove the assumption that the \CuSgp{} is countably based;
see \autoref{prp:SsoftDimGen}. 
The result applies, in particular, to the Cuntz semigroup of every \ca{} with the Global Glimm Property that has either strict comparison of positive elements, stable rank one, or topological dimension zero;
see \autoref{prp:dimOpt}.

We also study the dimension of the fixed-point algebra $A^\alpha$ for a finite group action $\alpha$; 
see \autoref{prp:DimWTRP}.

\begin{pgr}[Dimension of \CuSgp{s}]
Recall from \cite[Definition~3.1]{ThiVil22DimCu} that, given a \CuSgp{} $S$ and $n\in\NN$, we say that $S$ has \emph{dimension} $n$, in symbols $\dim (S)=n$, if $n$ is the least integer such that, whenever $x'\ll x\ll y_1+\ldots +y_r$, there exist elements $z_{j,k}\in S$ with $j=1,\ldots ,r$ and $k=0,\ldots ,n$ such that
\begin{itemize}
\item[(i)] 
$z_{j,k}\ll y_j$ for every $j$ and $k$;
\item[(ii)] 
$x'\ll \sum_{j,k}z_{j,k}$;
\item[(iii)] 
$\sum_{j}z_{j,k}\ll x$ for each $k$.
\end{itemize}

If no such $n$ exists, we say that $S$ has dimension $\infty$, in symbols $\dim (S)=\infty $.
\end{pgr}

The next result generalizes \cite[Proposition~3.17]{ThiVil22DimCu} to the nonsimple setting.

\begin{prp}
\label{prp:SsoftDim}
Let $S$ be a countably based, left-soft separative, $(2,\omega)$-divisible \CuSgp{} satisfying \axiomO{5}-\axiomO{7}, and assume that $\sigma\colon S\to S_\soft$ is a generalized \CuMor{}. 
Then,
\[
\dim(S_\soft)
\leq \dim(S)
\leq \dim(S_\soft)+1.
\]
\end{prp}
\begin{proof}
Since $\sigma$ is a generalized \CuMor{}, the first inequality follows from \cite[Proposition~3.15]{ThiVil22DimCu}.
To show the second inequality, set $n:=\dim(S_\soft)$, which we may assume to be finite.
To verify that $\dim(S)\leq n+1$, let $x'\ll x\ll y_1+\ldots+y_r$ in $S$.
We need to find $z_{j,k}\in S$ for $j=1,\ldots,r$ and $k=0,\ldots,n+1$ such that
\begin{itemize}
\item[(i)]
$z_{j,k}\ll y_j$ for each $j$ and $k$;
\item[(ii)]
$x'\ll\sum_{j,k}z_{j,k}$;
\item[(iii)]
$\sum_j z_{j,k}\ll x$ for each $k$.
\end{itemize}

First, choose $x'',x'''\in S$ such that $x'\ll x''\ll x''' \ll x$.
Applying that $S$ satisfies \axiomO{6} for $x''\ll x'''\leq y_1+\ldots+y_r$, we obtain $s_1,\ldots,s_r\in S$ such that
\[
x''\ll s_1+\ldots+s_r, \andSep s_j\ll x''',y_j \ \text{ for each } j=1,\ldots,r.
\]

Choose $s_1',\ldots,s_r' \in S$ such that
\[
x''\ll s_1'+\ldots+s_r', \andSep s_j'\ll s_j \ \text{ for each } j=1,\ldots,r.
\]

Using that $S$ is $(2,\omega)$-divisible (and consequently also $(r,\omega)$-divisible by \cite[Paragraph~2.4]{ThiVil23Glimm}), we obtain $v\in S$ such that
\[
rv\leq x, \andSep x'''\leq\infty v.
\]

For each $j$, we have $s_j \ll x''' \leq \infty v$.
Applying \cite[Proposition~4.10]{ThiVil23Glimm} to $s_j'\ll s_j \ll \infty v, \infty y_j$, we obtain $v_j\in S$ such that
\[
s_j' \ll \infty v_j, \andSep v_j\ll v,y_j.
\]

Note that
\[
x''
\ll s_1'+\ldots+s_r'
\leq \infty (v_1+\ldots+v_r), \andSep
v_1+\ldots+v_r \ll rv \leq x.
\]

Now, applying \autoref{prp:sigmaSoft} at the second step, we have
\[
x' \ll x'' \leq \sigma(x'')+(v_1+\ldots+v_r).
\]

Using that $S_\soft$ is a sub-\CuSgp{} by \autoref{prp:SoftPartAxioms}, we can choose an element $w\in S_\soft$ such that
\[
x' \ll w+(v_1+\ldots+v_r), \andSep
w\ll\sigma(x'').
\]

Applying that $\dim(S_\soft)\leq n$ for $w\ll\sigma(x'')\leq\sigma(y_1)+\ldots+\sigma(y_r)$, we obtain $z_{j,k}\in S_\soft$ for $j=1,\ldots,r$ and $k=0,\ldots,n$ such that
\begin{itemize}
\item[(i')]
$z_{j,k}\ll \sigma(y_j)$ for each $j$ and $k=0,\ldots,n$;
\item[(ii')]
$w\ll\sum_{j}\sum_{k=0}^n z_{j,k}$;
\item[(iii')]
$\sum_j z_{j,k}\ll \sigma(x'')$ for each $k=0,\ldots,n$.
\end{itemize}

Set $z_{j,n+1}:=v_j$ for each $j$.
These elements satisfy conditions~(i) and~(iii).
To verify~(ii), we note that
\[
x' 
\ll w+(v_1+\ldots+v_r) 
\ll (\sum_{j}\sum_{k=0}^n z_{j,k} )+(v_1+\ldots+v_r)
= \sum_{j}\sum_{k=0}^{n+1} z_{j,k},
\]
as desired.
\end{proof}

\begin{thm}
\label{prp:SsoftDimGen}
Let $S$ be a left-soft separative, $(2,\omega)$-divisible \CuSgp{} satisfying \axiomO{5}-\axiomO{7}. 
Additionally, assume one of the following:
\begin{itemize}
\item[(i)] 
$S$ is almost unperforated;
\item[(ii)] 
$S$ satisfies the Riesz Interpolation Property, and the interval axiom;
\item[(iii)] 
$S\otimes \{ 0,\infty \}$ is algebraic.
\end{itemize}
Then, $\dim(S_\soft)\leq \dim(S)\leq \dim(S_\soft)+1$.
\end{thm}
\begin{proof}
By \cite[Proposition~5.3]{ThiVil21DimCu2}, properties \axiomO{5}, \axiomO{6} and \axiomO{7} each satisfy the L\"{o}wenheim-Skolem condition.
Similarly, one can see that left-soft separativity, $(2,\omega )$-divisibility, and the properties listed in (i)-(iii) each satisfy the L\"{o}wenheim-Skolem condition. 
(For (iii), one can use \cite[Lemma~4.16]{ThiVil23Glimm}.) 
The proof is now analoguous to \cite[Proposition~5.9]{ThiVil21DimCu2} using \autoref{prp:SsoftDim}.
\end{proof}

\begin{cor}
\label{prp:dimOpt}
Let $A$ be a \ca{} with the Global Glimm Property. 
Additionally, assume one of the following:
\begin{itemize}
\item[(i)] 
$A$ has strict comparison of positive elements;
\item[(ii)] 
$A$ has stable rank one;
\item[(iii)] 
$A$ has topological dimension zero, and $\Cu (A)$ is left-soft separative.
\end{itemize}
Then, $\dim(\Cu (A)_\soft)\leq \dim(\Cu (A))\leq \dim(\Cu (A)_\soft)+1$.
\end{cor}
\begin{proof}
As in the proof of \autoref{prp:SigmaCuMorCStar}, we see that $\Cu(A)$ satisfies the corresponding assumptions of \autoref{prp:SsoftDimGen}, from which the result follows.
\end{proof}

\begin{ntn}
\label{ntn:CrossedProd}
Let $A$ be a \ca{}, and let $\alpha\colon G\to {\rm Aut}(A)$ be an action of a finite group $G$ on $A$. 
We will denote by $C^* (G,A,\alpha )$ the induced crossed product.
 
The \emph{fixed-point algebra} $A^\alpha$ is defined as
\[
A^\alpha := \big\{ a\in A : \alpha_g(a)=a\text{ for all }g\in G \big\}.
\]
\end{ntn}

\begin{pgr}[Fixed-point semigroups]
\label{pgr:FixedPointSgp}
For a group action $\alpha$ on a \ca{} $A$, there are three natural objects that may be seen as the fixed-point semigroup of $\Cu (A)$: The Cuntz semigroup $\Cu (A^\alpha)$, the fixed-point semigroup $\Cu (A)^\alpha$, and the fixed-point \CuSgp{} $\Cu (A)^{\Cu (\alpha )}$.
We give some details.
 
The \emph{fixed-point semigroup} $\Cu (A)^\alpha$ is defined as
\[
\Cu (A)^\alpha
:= \big\{ x \in \Cu (A) : \Cu(\alpha_g) (x) = x \text{ for all } g\in G \big\}.
\]
This is a submonoid of $\Cu(A)$ that is closed under passing to suprema of increasing sequences.
In general, it is not known if or when $\Cu (A)^\alpha$ is a sub-\CuSgp{} of~$\Cu(A)$.

An indexed collection $(x_t)_{t\in (0,1]}$ of elements in $S$ is a \emph{path} if $x_t\ll x_r$ whenever $r<t$ and $x_t = \sup_{r<t}x_r$ for every $t\in (0,1]$. The \emph{fixed-point \CuSgp{}}, as defined in \cite[Definition~2.8]{GarSan16EquivHomoRokhlin}, is 
\[
\Cu (A)^{\Cu (\alpha )} =\left\{
x\in\Cu (A) : 
\exists (x_t)_{t\in (0,1]}\text{ path in }\Cu (A) : 
\begin{array}{l}
x_1=x,\text{ and }\\
\Cu(\alpha_g) (x_t) = x_t \,\,\forall t, g
\end{array}\!\!
\right\}.
\]
Using \cite[Lemma~2.9]{GarSan16EquivHomoRokhlin}, one can show that $\Cu (A)^{\Cu (\alpha )}$ is always a sub-\CuSgp{} of $\Cu(A)$. Note that $\Cu (A)^{\Cu (\alpha )}$ is contained in $\Cu (A)^{\alpha}$. 
In \autoref{prp:FixedCuWTRP} we will see a situation in which $\Cu(A)^\alpha$ and $\Cu (A)^{\Cu (\alpha )}$ agree.
\end{pgr}

\begin{lma}
\label{prp:FixedSemigroupAxioms}
Let $S$ be an inf-semilattice ordered \CuSgp{}, and let $\alpha$ be an action of a finite group $G$ on $S$ by $\Cu$-isomorphisms on $S$.
Then the fixed-point semigroup $S^\alpha := \{ x \in S : \alpha_g(x)=x \text{ for all } g\in G \}$ is a sub-\CuSgp{} of $S$.

Moreover, if $S$ satisfies weak cancellation (resp. \axiomO{5}, \axiomO{6}, \axiomO{7}), then so does~$S^\alpha$.
\end{lma}
\begin{proof}
Define $\Phi \colon S \to S^\alpha$ by
\[
\Phi(x) := \bigwedge_{g \in G} \alpha_g(x)
\]
for $x \in S$.
For each $x \in S$, we have $\Phi(\Phi(x)) = \Phi(x) \leq x$;
and we have $\Phi(x)=x$ if and only if $x \in S^\alpha$.

It is straightforward to verify that $S^\alpha$ is a submonoid that is closed under suprema of increasing sequences.
To show that $S^\alpha$ is a sub-\CuSgp{}, it remains to verify that for given $x \in S^\alpha$ and $y \in S$ with $y \ll x$, there exists $x' \in S^\alpha$ with $y \leq x' \ll x$.

Let $(x_n)_n$ be a $\ll$-increasing sequence in $S$ with supremum $x$.
For each $g \in G$, we have $x = \alpha_g(x) = \sup_n \alpha_g(x_n)$, and it follows that
\[
x = \Phi(x) = \sup_n \Phi(x_n).
\]
Hence, there exists $n_0$ such that $y \leq \Phi(x_{n_0})$.
Set $x' := \Phi(x_{n_0})$.
Then $x' \in S^\alpha$ and
\[
y \leq x' \leq x_{n_0} \ll x,
\]
which shows that $x'$ has the desired properties.
Thus $S^\alpha$ is a sub-\CuSgp.

Since $S^\alpha$ is a sub-\CuSgp{} of $S$, it follows that $S^\alpha$ is weakly cancellative whenever $S$ is.
Assuming that $S$ satisfies \axiomO{5}, let us verify that so does $S^\alpha$.
Let $x',x,y',y,z \in S^\alpha$ satisfy
\[
x' \ll x, \quad
y' \ll y, \andSep
x+y \leq z.
\]
Choose $y'' \in S^\alpha$ satisfying $y' \ll y'' \ll y$.
Applying \axiomO{5} in $S$, we obtain $c \in S$ such that
\[
x' + c \leq z \leq x + c, \andSep
y'' \ll c.
\]
We claim that $\Phi(c)$ has the desired properties.
Indeed, for each $g \in G$, we have
\[
z 
= \alpha_g(z)
\leq \alpha_g(x+c)
= x + \alpha_g(c).
\]
Using that $S$ is semilattice-ordered, we get
\[
z 
\leq \bigwedge_{g \in G} \big( x + \alpha_g(c) \big)
= x + \bigwedge_{g \in G} \alpha_g(c)
= x + \Phi(c).
\]
We also have
\[
x'+\Phi(c) \leq x'+c \leq z, \andSep
y' \ll y'' = \Phi(y'') \leq \Phi(c).
\]

Assuming that $S$ satisfies \axiomO{6}, let us verify that so does $S^\alpha$.
Let $x',x,y,z \in S^\alpha$ satisfy
\[
x' \ll x \leq y + z.
\]
It suffices to find $\tilde{e} \in S^\alpha$ such that
\[
x ' \leq \tilde{e} + z, \andSep
\tilde{e} \leq x,y.
\]
(One can then apply this argument with the roles of $y$ and $z$ reversed to verify \axiomO{6}.)
Applying \axiomO{6} in $S$, we obtain $e \in S$ such that
\[
x' \leq e + z, \andSep
e \leq x,y.
\]

For each $g \in G$, we have
\[
x' = \alpha_g(x') \leq \alpha_g(e+z) = \alpha_g(e)+z.
\]
Using that $S$ is semilattice-ordered, we get
\[
x' \leq \bigwedge_{g \in G} \big( \alpha_g(e) + z \big)
= \left(\bigwedge_{g \in G} \alpha_g(c) \right) + z
= \Phi(e) + z.
\]
Further, we have
\[
\Phi(e) \leq e \leq x,y,
\]
which shows that $\tilde{e} := \Phi(e) \in S^\alpha$ has the desired properties.

Similarly, one shows that \axiomO{7} passes from $S$ to $S^\alpha$.
\end{proof}

We refer to \cite[Definition~2.2]{GarHirSan21RokDim} for the definition of the weak tracial Rokhlin property.
The first isomorphism in the statement below is well known, but we add it here for the convenience of the reader.

\begin{prp}
\label{prp:FixedCuWTRP}
Let $A$ be a non-elementary, stably finite, simple, unital \ca{}, and let $\alpha$ be a finite group action on $A$ that has the weak tracial Rokhlin property. 
Then we have 
\[
\Cu(C^*(G, A, \alpha)) 
\cong \Cu(A^\alpha), \andSep
\Cu(A)^{\Cu(\alpha)} 
= \Cu(A)^\alpha.
\]
Restricting to the soft parts, we obtain:
\[
\Cu(C^*(G, A, \alpha))_\soft
\cong \Cu(A^\alpha)_\soft 
\cong \Cu(A)^{\Cu(\alpha)}_\soft 
= \Cu(A)^\alpha \cap \Cu(A)_\soft.
\]

If, moreover, $A$ is separable and has stable rank one, then $\Cu (A)^\alpha$ is a simple, countably based, weakly cancellative, $(2,\omega )$-divisible sub-\CuSgp{} of $\Cu (A)$ satisfying \axiomO{5}-\axiomO{7}.
\end{prp}
\begin{proof}
For any action of a finite group on a unital \ca{}, the fixed-point algebra is $\ast$-isomorphic to a corner of the crossed product; 
see \cite[Lemma~4.3(4)]{AsaGolPhi21RadCompCrProd}. 
By \cite[Corollary~5.4]{HirOro13TraciallyZstable}, $C^*(G, A, \alpha)$ is simple, which implies that $C^*(G, A, \alpha)$ and~$A^\alpha$ are Morita equivalent and therefore have isomorphic Cuntz semigroups.

As noted in \autoref{pgr:FixedPointSgp}, $\Cu (A)^{\Cu (\alpha )}$ is contained in $\Cu(A)^\alpha$ in general, and $\Cu (A)^{\Cu (\alpha )}$ is always a sub-\CuSgp{} of~$\Cu(A)$. 
Let $\iota \colon A ^\alpha \to A$ denote the inclusion map, and note that $\Cu(\iota)$ takes image in $\Cu(A)^{\Cu(\alpha)}$.

To show that $\Cu(A)^\alpha$ is contained in $\Cu (A)^{\Cu (\alpha )}$, let $x \in \Cu (A)^\alpha$.
If $x$ is compact in $\Cu(A)$, then we can use the constant path $x_t = x$ to see that $x \in \Cu (A)^{\Cu(\alpha)}$.
On the other hand, if $x$ is soft, then we can apply \cite[Lemma~5.4]{AsaGolPhi21RadCompCrProd} to obtain $y \in \Cu(A^\alpha)_\soft$ such that $x = \Cu(\iota)(y)$.
Since $\Cu(\iota)$ takes image in $\Cu(A)^{\Cu(\alpha)}$, we have $x \in \Cu (A)^{\Cu(\alpha)}$.
Since $A$ is simple and stably finite, every Cuntz class is either compact or soft, and we have $\Cu(A)^{\Cu(\alpha)} = \Cu(A)^\alpha$.

We have shown
\[
\Cu(C^*(G, A, \alpha)) 
\cong \Cu(A^\alpha), \andSep
\Cu(A)^{\Cu(\alpha)} 
= \Cu(A)^\alpha.
\]

We know from \cite[Theorem~5.5]{AsaGolPhi21RadCompCrProd} that $\Cu(\iota)$ induces an order-isomorphism between the soft part of $\Cu(A^\alpha)$ and $\Cu(A)^\alpha \cap \Cu(A)_\soft$, the $\alpha$-invariant elements in $\Cu(A)_\soft$.
It is easy to see that $\Cu(\iota)$ maps $\Cu(A)_\soft$ into $\Cu(A)^{\Cu(\alpha)}_\soft$, and that $\Cu(A)^{\Cu(\alpha)}_\soft$ is contained in $\Cu(A)^\alpha \cap \Cu(A)_\soft$.
Together, we get 
\[
\Cu(A^\alpha)_\soft 
\xrightarrow[\Cu(\iota)]{\cong} \Cu(A)^{\Cu(\alpha)}_\soft 
= \Cu(A)^\alpha \cap \Cu(A)_\soft.
\]

Since $A^\alpha$ is a simple, nonelementary \ca{}, $\Cu (A^\alpha)$ is a simple, $(2,\omega)$-divisible \CuSgp{} satisfying \axiomO{5}-\axiomO{7}.
It follows from \autoref{prp:SoftPartAxioms} that $\Cu(A^\alpha)_\soft$ is a \CuSgp{} 
that also satisfies \axiomO{5}-\axiomO{7}.

\medskip

Finally, assume that $A$ is also separable and has stable rank one.
Then $\Cu(A)$ is a \CuSgp{} satisfying \axiomO{5}-\axiomO{7}.
Further, $\Cu(A)$ is weakly cancellative and inf-semilattice ordered by \cite[Theorem~4.3]{RorWin10ZRevisited} and \cite[Theorem~3.8]{AntPerRobThi22CuntzSR1}.
Hence, $\Cu(A)^\alpha$ satisfies \axiomO{5}-\axiomO{7} by \autoref{prp:FixedSemigroupAxioms}.

We have seen that $\Cu (A)^\alpha$ is a sub-\CuSgp{} of $\Cu(A)$.
Thus, since $\Cu (A)$ is simple and weakly cancellative, so is $\Cu (A)^\alpha$.
To verify $(2,\omega)$-divisibility, let $x \in \Cu (A)^\alpha$.
Since $A$ is simple and non-elementary, we know from \autoref{pgr:GGPNSCa} that $\Cu(A)$ is $(2,\omega)$-divisible.
Hence, there exists $y \in \Cu (A)$ such that $2y\leq x\leq\infty y$. 
Using \cite[Lemma~5.2]{AsaGolPhi21RadCompCrProd}, we find a nonzero element $z \in \Cu(A)^\alpha$ satisfying $z\leq y$. 
Then $2z \leq x \leq \infty z$, a priori in $\Cu(A)$, but then also in $\Cu(A)^\alpha$ since the inclusion $\Cu(A)^\alpha \to \Cu(A)$ is an order-embedding.
\end{proof}

\begin{thm}
\label{prp:DimWTRP}
Let $A$ be a non-elementary, separable, simple, unital \ca{} of stable rank one, and let $\alpha$ be a finite group action on $A$ that has the weak tracial Rokhlin property. 
Then
\begin{align}
\label{prp:DimWTRP:Eq1}
\dim \big( \Cu (C^*(G, A, \alpha) ) \big) 
= \dim \big(  \Cu (A^\alpha )  \big),
\end{align}
and
\begin{align*}
\dim \big( \Cu (A)^{\Cu(\alpha)} \big) -1
\leq \dim \big( \Cu (A^\alpha)  \big) 
\leq \dim \big( \Cu (A)^{\Cu(\alpha)}  \big) +1.
\end{align*}
\end{thm}
\begin{proof}
By \autoref{prp:FixedCuWTRP}, we have
\[
\Cu(C^*(G, A, \alpha)) 
\cong \Cu(A^\alpha),
\]
which immediately proves \eqref{prp:DimWTRP:Eq1}.

It also follows from \autoref{prp:FixedCuWTRP} that $\Cu (A)^{\Cu(\alpha)}$ is a simple, weakly cancellative (hence left-soft separative), $(2,\omega )$-divisible sub-\CuSgp{} of $\Cu (A)$ satisfying \axiomO{5}-\axiomO{7}.
Since $S$ is simple, $S\otimes\{0,\infty\}$ is algebraic.
(In fact, $S\otimes\{0,\infty\} \cong \{0,\infty\}$.)
Therefore, we can apply \autoref{prp:SsoftDimGen}~(iii) to obtain 
\begin{align*}
\dim \big( \Cu (A)^{\Cu(\alpha)}_\soft \big)
\leq \dim \big( \Cu (A)^{\Cu(\alpha)} \big)
\leq \dim \big( \Cu (A)^{\Cu(\alpha)}_\soft \big) + 1.
\end{align*}

Further, since $A^{\alpha }$ is simple and stably finite, we know from \cite[Remark~3.18]{ThiVil22DimCu} that 
\begin{align*}
\dim \big( \Cu (A^{\alpha})_\soft \big)
\leq \dim \big( \Cu (A^{\alpha}) \big)
\leq \dim \big( \Cu (A^{\alpha})_\soft \big) + 1.
\end{align*}

The result now follows since $\Cu (A^{\alpha})_\soft \cong \Cu (A)^{\Cu(\alpha)}_\soft$;
see \autoref{prp:FixedCuWTRP}.
\end{proof}

\begin{exa}
\label{exa:WTRP}
Let $n \geq 2$, and let $G$ be $S_n$, the symmetric group on the set $\{1,...,n\}$. Let
$A=\mathcal{Z}^{\otimes n}\cong \mathcal{Z}$, and let $\alpha \colon G \to \Aut (A)$ be the permutation action given by
\[
\alpha_{\theta}(a_1 \otimes a_2\otimes \ldots  \otimes a_n) = a_{\theta^{-1}(1)} \otimes a_{\theta^{-1}(2)} \otimes \ldots \otimes a_{\theta^{-1}(n)}.
\]

It follows from \cite[Example~5.10]{HirOro13TraciallyZstable} that $\alpha$ has the weak tracial Rokhlin property. 
Thus, using \autoref{prp:DimWTRP}, one has
\[
\dim \big(\Cu (A^\alpha)\big) 
=\dim \big(\Cu (C^*(G, A, \alpha))\big).
\]
 
The crossed product $\Cu (C^*(G, A, \alpha))$ is simple and $\mathcal{Z}$-stable; 
see Corollaries~5.4 and~5.7 from \cite{HirOro13TraciallyZstable}. 
Therefore, it follows from \cite[Proposition~3.22]{ThiVil22DimCu} that 
\[
\dim \big(\Cu (A^\alpha)\big) 
= \dim \big(\Cu (C^*(G, A, \alpha))\big)
\leq 1
\]
and, moreover, we have  $\dim ( \Cu (A)^{\Cu(\alpha)})\leq 2$ by \autoref{prp:DimWTRP}.
\end{exa}

\section{Radius of comparison of a Cuntz semigroup and its soft part}
\label{sec:RC}

In this section we show that, under the assumptions of \autoref{sec:SoftRanks}, the radius of comparison of a \CuSgp{} is equal to that of its soft part; 
see \autoref{prp:SsoftRc}. 
We deduce that the radius of comparison of a \ca{} $A$ is equal to that of the soft part of its Cuntz semigroup whenever $A$ is unital and separable, satisfies the Global Glimm Property, and has either stable rank one or strict comparison of positive elements; 
see \autoref{prp:SsoftRcCa}. 
This can be seen as a generalization of \cite[Theorem~6.14]{Phi14arX:LargeSub} to the setting of non-simple \ca{s};
see \autoref{rmk:RecoverPhiliips}.

We also study in \autoref{exa:rcWTRP} the radius of comparison of certain crossed products.

\begin{prp}
\label{prp:SepSoftRank}
Let $S$ be a countably based, left-soft separative, $(2,\omega )$-divisible \CuSgp{} satisfying \axiomO{5}-\axiomO{7}, and let $x\in S$. 
Then $\widehat{x}=\widehat{\sigma (x)}$. 
\end{prp}
\begin{proof} 
 By \autoref{prp:SoftRankS}, there exists $w \in S_\soft$ such that $w \leq x$ and $\widehat{x}=\widehat{w}$.
Since~$\sigma (x)$ is the largest strongly soft element dominated by~$x$ (\autoref{prp:sigmaSoft}), we get $w \leq \sigma(x)$, and so
\[
\widehat{x}=\widehat{w}\leq \widehat{\sigma (x)}\leq 
\widehat{x},
\]
as required.
\end{proof}

With the homeomorphism from \autoref{prp:FScongFSSoft} at hand, we can now relate the radius of comparison of $S$ and $S_\soft$. 
Let us first recall the definition of the radius of the comparison of \CuSgp{s} from Section~3.3 of \cite{BlaRobTikTomWin12AlgRC}. 

\begin{dfn}
\label{dfn:RadCom}
Given a \CuSgp{} $S$, a full element $e \in S$ and $r>0$, one says that the pair $(S,e)$ satisfies condition (R1) for $r$ if $x,y \in S$ satisfy $x \leq y$ whenever
\[
\lambda (x) + r\lambda (e)\leq \lambda (y)
\]
for all $\lambda \in F(S)$.

The \emph{radius of comparison} of $(S,e)$, denoted by $\rc (S, e)$, is the infimum of the positive elements $r$ such that $(S,e)$ satisfies (R1) for $r$.
\end{dfn}

\begin{rmk}
\label{rmk:RadCom}
In \cite[Deﬁnition~3.3.2]{BlaRobTikTomWin12AlgRC}, for a \ca{} $A$ and a full element $a \in (A \otimes \mathcal{K})_+$, the notation $r_{A, a}$ is used for $\rc (\Cu (A),[a])$. 
Also, it was shown in \cite[Proposition~3.2.3]{BlaRobTikTomWin12AlgRC} that for unital \ca{s} all of whose quotients are stably finite, the radius of comparison $\rc (\Cu (A),[ 1_A])$ coincides with the original notion of radius of comparison $\rc (A)$ as introduced in \cite[Definition~6.1]{Tom06FlatDimGrowth}.
\end{rmk}

\begin{prp}
\label{prp:RadComGenCuMor}
Let $\varphi \colon S \to T$ be a generalized \CuMor{} between \CuSgp{s} that is also an order embedding, and let $e \in S$ be a full element such that~$\varphi(e)$ is full in $T$.
Then, $\rc (S, e) \leq \rc (T, \varphi(e))$.
\end{prp}
\begin{proof}
Take $r>0$. We show that $(S,e)$ satisfies condition (R1) for $r$ whenever $(T, \varphi(e))$ does, which readily implies the claimed inequality.

Thus, assume that $(T, \varphi(e))$ satisfies condition (R1) for $r$. 
In order to verify that $(S, e)$ satisfies (R1) for $r$ as well, let $x,y \in S$ satisfy 
\[
\lambda (x)+r\lambda (e) \leq \lambda (y)
\]
for all $\lambda \in F(S)$.

Note that, for every $\rho \in F(T)$, we have that $\rho \circ \varphi \in F(S)$. 
Thus, we get 
\[
\rho (\varphi(x))+r\rho (\varphi(e)) \leq \rho (\varphi(y))
\]
for every $\rho \in F(T)$.
It follows from our assumption that $\varphi (x) \leq \varphi (y)$ and, since $\varphi$ is an order-embedding, we deduce that $x \leq y$, as desired.
\end{proof}

\begin{thm}
\label{prp:SsoftRc}
Let $S$ be a $(2,\omega)$-divisible \CuSgp{} satisfying \axiomO{5}-\axiomO{7}, and let $e\in S$ be a full element. 
Then, there exists $w \in S_\soft$ such that 
\[
\rc (S,e) = \rc (S_\soft , w), \quad
w\leq e \leq \infty w, \andSep 
\widehat{e} = \widehat{w}.
\]

If $S$ is also countably based and left-soft separative, we have 
\[
\rc (S,e) = \rc (S_\soft , \sigma (e)).
\]
\end{thm}
\begin{proof}
By \autoref{prp:SoftRankS}, we can pick $w \in S_\soft$ such that 
\[
w \leq e \leq \infty w,
\andSep \widehat{e} = \widehat{w}.
\]
 
Using at the first step that the inclusion map $\iota \colon S_\soft \to S$ is a \CuMor{} and an order-embedding and applying \autoref{prp:RadComGenCuMor}, and using at the last step that $\widehat{e}=\widehat{w}$, we get 
\[
\rc (S_\soft, w)
\leq \rc (S, \iota (w))
= \rc(S, w)
= \rc(S, e).
\]

To prove the converse inequality, let $r>0$ and assume that $(S_\soft, w)$ satisfies condition (R1) for~$r$.
Take $\varepsilon >0$.
We will show that $(S,e)$ satisfies (R1) for $r+\varepsilon$.
  
Now let $x,y\in S$ be such that $\lambda (x)+(r+\varepsilon)\lambda (e)\leq \lambda (y)$ for every $\lambda \in F(S)$ or, equivalently, such that
\[
\widehat{x}+(r+\varepsilon )\widehat{e}\leq \widehat{y}
\]
in $\LAff (F(S))$.

Applying \cite[Proposition~7.7]{ThiVil23arX:Soft}, we find $k\in\NN$ and then $t \in S_\soft$ such that 
\[
kt \leq e \leq \infty t,\andSep 
1\leq k\varepsilon .
\]

Thus, we get 
\[
\widehat{x+t}+r\widehat{e} 
\leq \widehat{x}+k\varepsilon\widehat{t}+r\widehat{e}
\leq \widehat{x}+\varepsilon\widehat{e}+r\widehat{e} 
= \widehat{x} + (\varepsilon +r)\widehat{e}
\leq \widehat{y}.
\]

Note that, since $e$ is full in $S$, so is $t$. 
By \cite[Theorem~4.14(2)]{ThiVil23arX:Soft}, this implies that $x+t$ is strongly soft.

By \autoref{prp:SoftRankS}, there exists $v\in S_\soft$ such that $v\leq y$ and $\widehat{v}=\widehat{y}$. 
One gets
\[
\widehat{x+t}+r\widehat{w}
= \widehat{x+t}+r\widehat{e}
\leq \widehat{y}
= \widehat{v}
\]
or, equivalently, that 
\[
\lambda(x+t) + r\lambda(w)
\leq \lambda(v)
\]
for every $\lambda \in F(S)$.

Using that $F(S)\cong F(S_\soft)$ (\autoref{prp:FScongFSSoft}) and that $(S_\soft, w)$ satisfies condition~(R1) for $r$, it follows that 
\[
x \leq x+t \leq v \leq y.
\]

This shows that, given any $\varepsilon >0$, $(S,e)$ satisfies condition~(R1) for $r+\varepsilon$ whenever $(S_\soft,w)$ satisfies~(R1) for $r$.
Consequently, we have $\rc(S,e) \leq \rc(S_\soft, w)$, as required.

Finally, if $S$ is also countably based and left-soft separative, then we can use $w:=\sigma(x)$ by \autoref{prp:SepSoftRank}.
\end{proof}

\begin{thm}
\label{prp:SsoftRcCa}
Let $A$ be a unital, separable \ca{} with the Global Glimm Property. 
Assume that $A$ has stable rank one. 
Then 
\[
\rc \big( \Cu(A),[1] \big) 
= \rc \big( \Cu(A)_\soft , \sigma ([1]) \big).
\]
\end{thm}
\begin{proof}
Proceeding as in the proof of \autoref{prp:SigmaCuMorCStar}, we see that the assumptions on $A$ imply that $\Cu(A)$ is a countably based, left-soft separative, $(2,\omega)$-divisible \CuSgp{} satisfying \axiomO{5}-\axiomO{7}, and that $[1]$ is full.
Hence, the result follows from \autoref{prp:SsoftRc}.
\end{proof}


\begin{cor}
\label{prp:RcCa}
Let $A$ be a unital, separable, nowhere scattered \ca{} of stable rank one. 
Then
\[
\rc(A) =  \rc \big(\Cu(A)_\soft ,\sigma ([ 1 ] ) \big).
\]
\end{cor}
\begin{proof}
By \cite[Proposition~7.3]{ThiVil23Glimm}, $A$ has the Global Glimm Property:
see also \cite[Section~5]{AntPerRobThi22CuntzSR1}.
Further, by \cite[Proposition~3.2.3]{BlaRobTikTomWin12AlgRC}, we have $\rc(A) = \rc \big( \Cu(A),[1] \big)$, and so the result follows from \autoref{prp:SsoftRcCa}.
\end{proof}

\begin{rmk}
\label{rmk:RecoverPhiliips}
For a large subalgebra $B$ of a simple, unital, stably finite, non-elementary \ca{} $A$, it is shown in \cite[Theorem~6.8]{Phi14arX:LargeSub} that $\Cu(A)_\soft\cong \Cu (B)_\soft$; see also \autoref{pgr:SoftSubCu}. Thus, using \autoref{prp:SsoftRc} at the first and last steps, one gets
\[
\rc (A)=\rc (\Cu (A)_\soft,\sigma_A([1]))=\rc (\Cu (B)_\soft,\sigma_B([1]))=\rc (B),
\]
which recovers \cite[Theorem~6.14]{Phi14arX:LargeSub}.

Note that in this case the existence of $\sigma$ is provided by \cite{Eng14PhD}. 
\end{rmk}

\begin{exa}
\label{exa:rcWTRP}
Let $A$ be a non-elementary, separable, simple, unital \ca{} of stable rank one, real rank zero, and such that the order of projections over $A$ is determined by traces, and let $\alpha$ be a finite group action on $A$ that has the tracial Rokhlin property. 
Then
\[
\rc \big( \Cu(A^\alpha) ,[1] \big)
= \rc \big( \Cu (A)^{\Cu(\alpha)} ,[1] \big).
\]

Indeed, by \cite{Arc11CrProdTrRP}, the crossed product $C^*(G, A, \alpha)$ has stable rank one, and then so does the fixed point algebra $A^\alpha$ by \cite[Lemma~4.3]{AsaGolPhi21RadCompCrProd}.
The question of when stable rank one passes to crossed producst by a finite group action with the (weak) tracial Rokhlin property is discussed after Corollary~5.6 in \cite{AsaGolPhi21RadCompCrProd}.
One can also see that $A^\alpha$ is non-elementary, separable, simple and unital.
Therefore, $\Cu(A^\alpha)$ is a countably based, weakly cancellative (hence, left-soft separative), $(2,\omega )$-divisible \CuSgp{s} satisfying \axiomO{5}-\axiomO{7}.
By \autoref{prp:FixedCuWTRP}, the \CuSgp{} $\Cu(A)^{\Cu(\alpha)}$ has the same properties.
Further, the soft parts of $\Cu(A^\alpha)$ and $\Cu(A)^{\Cu(\alpha)}$ are isomorphic by \autoref{prp:FixedCuWTRP}.

This allows us to apply \autoref{prp:SsoftRc} at the first and last steps, and we get
\begin{align*}
\rc \big( \Cu(A^\alpha) ,[1] \big)
&= \rc \big( \Cu(A^\alpha)_\soft ,\sigma([1]) \big) \\
&= \rc \big( \Cu(A)^{\Cu(\alpha)}_\soft,\sigma([1]) \big) 
= \rc \big( \Cu(A)^{\Cu(\alpha)} ,[1] \big).
\end{align*}

Other examples where our results might be applicable are those obtained in \cite{Asa23RadCompCrossedPro}.
\end{exa}




\providecommand{\etalchar}[1]{$^{#1}$}
\providecommand{\bysame}{\leavevmode\hbox to3em{\hrulefill}\thinspace}
\providecommand{\noopsort}[1]{}
\providecommand{\mr}[1]{\href{http://www.ams.org/mathscinet-getitem?mr=#1}{MR~#1}}
\providecommand{\zbl}[1]{\href{http://www.zentralblatt-math.org/zmath/en/search/?q=an:#1}{Zbl~#1}}
\providecommand{\jfm}[1]{\href{http://www.emis.de/cgi-bin/JFM-item?#1}{JFM~#1}}
\providecommand{\arxiv}[1]{\href{http://www.arxiv.org/abs/#1}{arXiv~#1}}
\providecommand{\doi}[1]{\url{http://dx.doi.org/#1}}
\providecommand{\MR}{\relax\ifhmode\unskip\space\fi MR }
\providecommand{\MRhref}[2]{%
  \href{http://www.ams.org/mathscinet-getitem?mr=#1}{#2}
}
\providecommand{\href}[2]{#2}

\end{document}